\documentclass[12pt]{amsart}

\usepackage{latexsym}
\usepackage{amsmath}
\usepackage{amssymb}
\usepackage{amsthm}
\usepackage{amscd}
\usepackage{color}
\usepackage{graphics}

\theoremstyle{plain}
\newtheorem{theorem}{Theorem}
\newtheorem{proposition}{Proposition}

\newtheorem{lemma}{Lemma}[section] 

\theoremstyle{definition}
\newtheorem{remark}{Remark}
 \numberwithin{equation}{section}

\def\C{\mathbb{C}}
\def\R{\mathbb{R}}
\def\Rn{{\mathbb{R}^n}}
\def\Sn{{{S}^{n-1}}}

\def\N{\mathbb{N}}

\def\Z{\mathbb{Z}}
\def\F{\mathcal{F}}

\def\S{\mathcal{S}}
\def\supp{\operatorname{supp}}

\allowdisplaybreaks
\begin{document}

\title [Fractional type Marcinkiewicz integral operators]{Fractional type
Marcinkiewicz integral operators associated to surfaces}
\author{
Yoshihiro Sawano and K\^{o}z\^{o} Yabuta}
\keywords{ 
$L^p$ boundedness; Marcinkiewicz integral; fractional integral operator;
Triebel-Lizorkin spaces; Sobolev spaces}

\subjclass[2000]{ 
Primary 42B20; Secondary 42B25, 47G10
}
\date{\today}
\thanks{
This work was partially supported by
Grant-in-Aid for Scientific Research (C) No.~23540228,
Japan Society for the Promotion of Science
and
Grant-in-Aid for Young Scientists (B) No.~24740085,
Japan Society for the Promotion of Science..
}
\address{Yoshihiro Sawano \endgraf Department of Mathematics and Information
Science
\endgraf Tokyo Metropolitan University \\Minami-Ohsawa 1-1, Hachioji 192-0397
\endgraf Japan }
\address{K\^{o}z\^{o} Yabuta \endgraf Research Center for Mathematical
Sciences \endgraf Kwansei Gakuin University\\Gakuen 2-1, Sanda 669-1337
\endgraf Japan }

\date{\today}
\maketitle
\begin{abstract} 
In this paper, 
we discuss the boundedness 
of the fractional type Marcinkiewicz integral operators 
associated to surfaces, and 
extend a result given 
by Chen, Fan and Ying in 2002. 
They showed that 
under certain conditions the fractional type Marcinkiewicz 
integral operators are bounded from the Triebel-Lizorkin spaces 
$\dot F_{pq}^{\alpha}(\Rn)$ to $L^p(\Rn)$. 
Recently the second author,
together with Xue and Yan, 
greatly weakened their assumptions. 
In this paper, we extend their results to the case 
where the operators are associated to the surfaces 
of the form 
$\{x=\phi(|y|)y/|y|\} \subset \Rn \times (\Rn \setminus \{0\})$. 

To prove our result, 
we discuss a characterization of the homogeneous 
Triebel-Lizorkin spaces  
in terms of lacunary sequences. 
\end{abstract}

\section{Introduction}
The fractional type Marcinkiewicz operator
is defined by
\begin{equation}\label{def:fracMar-1}
\mu_{\Omega,\rho,\alpha}f(x)
=\biggl(\int_{0}^{\infty}\biggl|\frac{1}{t^{\rho+\alpha}}
\int_{B(t)}f(x-y)\frac{\Omega(y/|y|)}{|y|^{n-\rho}}dy\biggr|^2\frac{dt}{t}
\biggr)^{\frac{1}{2}},
\end{equation}
where we write $B(r)=\{|x|<r\} \subset \Rn$ for $r>0$
here and below.
The operator
$\mu_{\Omega,\rho,\alpha}f$
is the so called singular integral operator.
In this paper, we shall prove that this operator
is bounded under a certain highly weak integrability assumption.
To this end,
we plan to employ a modified Littlewood-Paley decomposition
adapted to our situation.
It turns out that we can relax the integrability assumption on $\Omega$
and that the integral operator itself can be generalized
to a large extent.

Let $\Sn$ be the unit sphere in the $n$-dimensional Euclidean space $\Rn$
$(n\ge2)$, with the induced Lebesgue measure $d\sigma=d\sigma(x')$ and 
$\Omega\in L^1(\Sn)$. In the sequel, we often suppose that $\Omega$
satisfies the cancellation condition
\begin{equation}\label{eq:cancellation}
\int_{\Sn}\Omega(y')\,d\sigma(y')=0.
\end{equation}
Here, for the symbols $x'$ and $y'$,
we adopt the following convention:
Sometimes they stand for points in $\Sn$.
But for $x \in \Rn\setminus\{0\}$, we abbreviate
$x/|x|$ to $x'$
in the present paper.
We make this slight abuse of notation
since no confusion is likely to occur.

In the present paper we deal
with operators of Marcinkiewicz type.
Define
\begin{equation}\label{def:fracMar-q}
\mu_{\Omega,\rho,\alpha,q}f(x)
:=\biggl(\int_{0}^{\infty}\biggl|\frac{1}{t^{\rho+\alpha}}
\int_{B(t)}f(x-y)\frac{\Omega(y')}{|y|^{n-\rho}}dy
\biggr|^q\frac{dt}{t}\biggr)^{\frac{1}{q}}
\quad(x\in\Rn).
\end{equation}
As a special case,
by letting $\rho=1,\,\alpha=0,\,q=2$,
we recapture
the Marcinkiewicz integral  operator that 
E.~M.~Stein  introduced in 1958 \cite{Stein1}. 
In 1960,
H\"ormander considered 
the parametric Marcinkiewicz integral operator $\mu_{\Omega,\rho,\alpha,2}$ 
\cite{Hormander1}. 
Since then, about Marcinkiewicz type integral operators, 
many works appeared. 
A nice survey is given by S. Lu \cite{Lu1}.

We formulate our results
in the framework of Triebel-Lizorkin spaces of homogeneous type.
For $\alpha \in \R$ and $p,q \in (1,\infty)$,
we let
$\dot F_{pq}^{\alpha}(\Rn)$
be the Triebel-Lizorkin space
defined in \cite{Tr2}.
Note that the space ${\mathcal S}_\infty(\Rn)$ given by
\[
{\mathcal S}_\infty(\Rn)
:=
\bigcap_{\alpha \in ({\mathbb N} \cup \{0\})^n}
\left\{f \in {\mathcal S}(\Rn)\,:\,
\int_{\Rn}x^\alpha f(x)\,dx=0\right\}
\]
is dense in 
$\dot F_{pq}^{\alpha}(\Rn)$
as long as $\alpha \in \R$ and $p,q \in (1,\infty)$.
If $u \in (1,\infty)$,
then define
$u'=\frac{u}{u-1}$
and
$\tilde{u}=\max(u,u')$.
Here and below 
a tacit understanding in the present paper
is that the letter $C$ is used
for constants that may change from one occurrence to another,
that is, the letter $C$ will denote a positive constant which may
vary from line to line but will remain independent of the relevant quantities.
Our main theorem in the simplest form reads as follows:
\begin{theorem}\label{thm:fracMarc-1}
Let $\rho>0$, $1<p,\, q<\infty$, and $\Omega\in L^1(\Sn)$. 
\begin{enumerate}
\item[{\rm(i)}]
If $\alpha\in(0,4/(\tilde{p}\tilde{q}))$ 
and $\Omega$ satisfies the cancellation condition \eqref{eq:cancellation}, 
then
\begin{equation}\label{eq:fracMar-0-3}
\|\mu_{\Omega,\rho,\alpha,q}f\|_{L^p(\Rn)}
\le C 
\|\Omega\|_{L^1(\Sn)}\|f\|_{\dot F_{pq}^{\alpha}(\Rn)}
\end{equation}
for all $f \in {\mathcal S}_\infty(\R^n)$.
\item[{\rm(ii)}]
If $\alpha\in
\left(-\min\{\frac{4\beta}{\tilde{p}\tilde{q}},\,\rho\},0\right)$, and
\begin{equation}\label{eq:fracMar-0-4}
Z_\Omega
:=
\sup_{\xi'\in \Sn}\int_{\Sn}\frac{|\Omega(y')|}{|y'\cdot\xi'|^\beta}\,
d\sigma(y')<+\infty,
\end{equation}
for some  $0<\beta\leq1$,
then
\begin{equation}\label{eq:fracMar-0-3-19}
\|\mu_{\Omega,\rho,\alpha,q}f\|_{L^p(\Rn)}
\le C 
Z_\Omega\|f\|_{\dot F_{pq}^{\alpha}(\Rn)}
\end{equation}
for all $f \in {\mathcal S}_\infty(\R^n)$.
\item[{\rm(iii)}]
If $\alpha=0$ and $\Omega\in L\log L(\Sn)$ satisfies the cancellation 
condition \eqref{eq:cancellation}, 
then 
\begin{equation}\label{eq:fracMar-0-3-1}
\|\mu_{\Omega,\rho,\alpha,q}f\|_{L^p(\Rn)}
\le C 
\|\Omega\|_{L\log L(\Sn)}\|f\|_{\dot F_{pq}^{\alpha}(\Rn)}
\end{equation}
for all $f \in {\mathcal S}_\infty(\R^n)$.
\end{enumerate}
In any case,
by density we can extend
$(\ref{eq:fracMar-0-3})$,
$(\ref{eq:fracMar-0-3-19})$
and
$(\ref{eq:fracMar-0-3-1})$
and have them
for all $f \in \dot{F}_{pq}^{\alpha}(\Rn)$.
\end{theorem}
In 2002, J.~Chen, D.~Fan and Y.~Ying obtained a result
about 
the fractional type Marcinkiewicz integral operator  \cite{ChenFanYing},
which we recall now.
\par\medskip\noindent
{\bf Theorem A.} {\it
Let $1<p,\, q<\infty$ and $1<r\le\infty$. Suppose $\Omega\in L^r(\Sn)$
satisfies the cancellation condition \eqref{eq:cancellation}.
If $|\alpha|<2/(r'\tilde{p}\tilde{q})$ and $\rho=1$, then
\[
\|\mu_{\Omega,\rho,\alpha,q}f\|_{L^p(\Rn)}
\le C 
\|\Omega\|_{L^r(\Sn)}\|f\|_{\dot F_{pq}^{\alpha}(\Rn)}
\]
for all $f \in {\mathcal S}_\infty(\Rn)$.
}\par\medskip
Si, Wang and Jiang
discussed ones of somewhat different type
\cite{SiWangJiang}.
About Theorems \ref{thm:fracMarc-1} and A,
a couple of remarks may be in order.
\begin{remark}
If $0<\beta<1$, $1/(1-\beta)<r\le\infty$ and $\Omega\in L^r(\Sn)$, 
it is easily seen that the condition \eqref{eq:fracMar-0-4} is satisfied. 
In this case
\[
Z_\Omega \le C\|\Omega\|_{L^r(\Sn)}.
\] 
So, our
result includes completely Theorem A,
where they assumed that $\Omega\in L^r(\Sn)$. 
Let $r>1$ and define
\begin{equation}
\Omega_0(y')
=\operatorname{sgn}(y'\cdot(1,0,\dots,0)) 
|(y'\cdot(1,0,\dots,0))|^{-1/r}.
\end{equation} 
Then, it is also easily checked that 
$\Omega$ is in $L^1(\Sn) \setminus L^r(\Sn)$ 
and satisfies \eqref{eq:fracMar-0-4} for any $\beta \in (0,1/r')$.\par
In the case $\alpha=0$, $\rho=1$ and $q=2$,
 the conclusion in Theorem \ref{thm:fracMarc-1}{\rm(iii)} 
is shown to hold even when $\Omega\in L\log L^{1/2}(\Sn)$
 in \cite{AACPan}.
\end{remark}
\begin{remark}
We can relax the condition on $\alpha$: 
$|\alpha|<4/(r'\tilde{p}\tilde{q})$ suffices.
Indeed, one can get
$|(\Omega(\cdot)|\cdot|^{-n+1}\chi_{B(1)})\hat{}(\xi)|
\le C|\xi|^{-1/r'}$ by direct computation.

By reexamining their proof, 
we can parametrize Theorem A:
we can prove
\begin{equation}\label{eq:fracMar-0-2}
\biggl(\int_{\Rn}\biggl(\int_{0}^{\infty}\biggl|\frac{1}{t^{\rho+\alpha}}
\int_{B(t)}f(x-y)\frac{\Omega(y)}{|y|^{n-\rho}}dy\biggr|^q\frac{dt}{t}
\biggr)^{{p}/{q}}dx\biggr)^{1/p}
\le C \|f\|_{\dot F_{pq}^{\alpha}(\Rn)},
\end{equation}
provided $|\alpha|<4\min\left\{\frac1{r'},\min(\rho,1)\right\}
\frac{1}{\tilde{p}\tilde{q}}$.
Comparing (\ref{eq:fracMar-0-2}) with Theorem \ref{thm:fracMarc-1},
one concludes that our theorem outranges Theorem A
in view of 
the case when $\min(\rho,1)<1/r'$.
In our earlier paper \cite{XueYabutaYan1}, 
we improved Theorem A by relaxing the conditions postulated on $\Omega$. 
\end{remark}

Our method is also applicable even in more generalized settings.
For $\rho>0$, $\alpha\in \R$ and $\Omega\in L^1(\Sn)$, we define the fractional
 type Marcinkiewicz integral operator by (\ref{def:fracMar-1})
and the fractional type Marcinkiewicz integral operator associated to surfaces 
$\{(x,y)\,:\,x=\phi(|y|)y'\} \in \Rn \times (\Rn \setminus \{0\})$ by
\begin{equation}\label{def:fracMarSurf-1}
\mu_{\Omega,\rho, \phi,\alpha}f(x)
=\biggl(\int_{0}^{\infty}\biggl|\frac{1}{t^{\rho}\phi(t)^\alpha}
\int_{B(t)}f(x-\phi(|y|)y')\frac{\Omega(y)}{|y|^{n-\rho}}dy\biggr|^2
\frac{dt}{t}\biggr)^{\frac{1}{2}}.
\end{equation}

Theorem \ref{thm:fracMarc-1} extends further 
to the case when the operator
is equipped with a function space $\Delta_\gamma$ with $\gamma \ge 1$.
 Regarding to Calder\'on-Zygmund singular
integral and Marcinkiewicz integral operators, many authors discussed those
operators with modified kernel $b(|\cdot|)\Omega(\cdot)$ in place of
$\Omega(\cdot)$, where $b$ belongs to the class of all measurable functions
$h:[0,\infty)\to\C$ satisfying $\|h\|_{\Delta_\gamma}=
\sup_{R>0}\bigl(R^{-1}\int_0^R|h(t)|^\gamma dt\bigr)^{1/\gamma}<\infty$
$(1\le \gamma\le\infty)$, 
see \cite{QassemPan,DuoRubio,FanPan1,FanPan2,RFefferman}, etc. 
We note that
\begin{equation*}
L^\infty(\R_+)\subset \Delta_\beta(\R_+)\subset \Delta_\alpha(\R_+)\quad\text{
for }1\le\alpha<\beta,
\end{equation*}
and that all these inclusions are proper. 
We refer to \cite{BGST,KMNS,MN} for extension and generalization
of the space $\Delta_\gamma$.


We define the modified fractional type Marcinkiewicz operator
$\mu_{\Omega,\rho,\alpha,q}^{(b)}$ by
\begin{equation}\label{def:fracMar-b-q}
\mu_{\Omega,\rho,\alpha,q}^{(b)}f(x)
=\biggl(\int_{0}^{\infty}\biggl|\frac{1}{t^{\rho+\alpha}}
\int_{B(t)}f(x-y)\frac{b(|y|)\Omega(y')}{|y|^{n-\rho}}dy
\biggr|^q\frac{dt}{t}\biggr)^{\frac{1}{q}}.
\end{equation}
We can recover Theorem \ref{thm:fracMarc-1}
by letting $b \equiv 1$ in the next theorem.
\begin{theorem}\label{thm:fracMarc-1-b}
Suppose that we are given
$\Omega\in L^1(\Sn)$
and parameters 
$p,q,\alpha,\gamma,\rho$
satisfying
$$
1<p,\, q<\infty,
\gamma>\frac12\max\{\tilde{p},\tilde{q}\},
\rho>0.
$$
\begin{enumerate}
\item[{\rm(i)}]
Let
$
\alpha\in\left(0,
\frac{4(1/\tilde{p}-1/(2\gamma))(1/\tilde{q}-1/(2\gamma))}{(1-1/\gamma)^2}
\right).
$
If $b\in\Delta_\gamma(\R_+)$ and $\Omega$ satisfies the 
cancellation condition \eqref{eq:cancellation}, 
then
\begin{equation}\label{eq:fracMar-0-3-b}
\|\mu_{\Omega,\rho,\alpha,q}^{(b)}f\|_{L^p(\Rn)}
\le C 
\|\Omega\|_{L^1(\Sn)}\|b\|_{\Delta_\gamma}\|f\|_{\dot F_{pq}^{\alpha}(\Rn)}
\end{equation}
for all $f \in {\mathcal S}_\infty(\R^n)$.
\item[{\rm(ii)}]
Assume 
$
\alpha
\in\left(-\min\left\{2\beta\frac{1/\tilde{p}-1/(2\gamma)}{1-1/\gamma}
\cdot
\frac{1/\tilde{q}-1/(2\gamma)}{1-1/\gamma},\,\rho\right\},0
\right)
$
with $\beta \in (0,1]$.
If $b \in \Delta_{\max(\gamma,2)}$
and
\begin{equation}\label{eq:fracMar-0-4-b}
W_\Omega:=
\sqrt{\sup_{\xi'\in \Sn}\int_{\Sn\times\Sn}
\frac{|\Omega(y')\Omega(z')|}{|(y'-z')\cdot\xi'|^\beta}
\,d\sigma(y')d\sigma(z')}
<+\infty,
\end{equation}
then 
\begin{equation}\label{eq:fracMar-0-3-b-101}
\|\mu_{\Omega,\rho,\alpha,q}^{(b)}f\|_{L^p(\Rn)}
\le C 
W_\Omega\|b\|_{\max(\gamma,2)}\|f\|_{\dot F_{pq}^{\alpha}(\Rn)}
\end{equation}
for all $f \in {\mathcal S}_\infty(\R^n)$.
\item[{\rm(iii)}]
Assume $\alpha=0$.
If 
$b \in \Delta_{\max(\gamma,2)}$,
$\Omega\in L\log L(\Sn)$
and
$\Omega$ satisfies the cancellation condition 
\eqref{eq:cancellation}, then 
\begin{equation}\label{eq:fracMar-0-3-b-10}
\|\mu_{\Omega,\rho,\alpha,q}^{(b)}f\|_{L^p(\Rn)}
\le C 
\|\Omega\|_{L\log L(\Sn)}
\|b\|_{\max(\gamma,2)}\|f\|_{\dot F_{pq}^{\alpha}(\Rn)}
\end{equation}
for all $f \in {\mathcal S}_\infty(\R^n)$.
\end{enumerate}
In any case,
by density we can extend
$(\ref{eq:fracMar-0-3-b})$,
$(\ref{eq:fracMar-0-3-b-101})$
and
$(\ref{eq:fracMar-0-3-b-10})$
and have them
for all $f \in \dot{F}_{pq}^{\alpha}(\Rn)$.
\end{theorem}
\begin{remark}\label{rm:fracMarc-1-b}
In Theorem \ref{thm:fracMarc-1}(ii)
a modification of the proof changes
$4\beta$ into $2\beta$.
We cannot estimate directly the Fourier transform of the measure
${\sigma_t}$ in Section 3, and use the idea given by Duoandikoetxea and
Rubio~de~Francia \cite[p.~551]{DuoRubio} as in Chen, Fan and Ying
\cite{ChenFanYing}.\par
If $0<\beta<1$, $1/(1-\beta)<r\le\infty$ and $\Omega\in L^r(\Sn)$, 
it is easily seen that the condition \eqref{eq:fracMar-0-4-b} is satisfied. 
In this case
\[
W_\Omega \le C\|\Omega\|_{L^r(\Sn)}.
\] 
In the case $\alpha=0$, $\rho=1$ and $q=2$, 
it is again known in \cite{Al-Qassem1}
that the conclusion in Theorem \ref{thm:fracMarc-1-b}{\rm(iii)} holds 
even when $\Omega\in L\log L^{1/2}(\Sn)$.\par
In the earlier paper \cite{XueYabutaYan1},  
in Theorem \ref{thm:fracMarc-1}(ii) 
(respectively in Theorem \ref{thm:fracMarc-1-b}(ii)),
we needed to postulate
the additional conditions $\rho>\beta$ 
(respectively $2\rho>\beta$) and the cancellation condition on $\Omega$.
However, these are no longer necessary in the new theorems.
\end{remark}

In addition to the factor of $b$,
we can even distort the convolution.
For $\alpha>0$, $1\le q<\infty$,
a kernel $\Omega$ and a positive function $\phi$ on $\R_+$, we define 
the operator $\mu_{\Omega,\rho,\phi,\alpha,q}$ 
and the modified one $\mu_{\Omega,\rho,\phi,\alpha,q}^{(b)}$ by
\begin{equation}\label{def:fracMarSurf-q}
\mu_{\Omega,\rho,\phi,\alpha,q}f(x)
=\biggl(\int_{0}^{\infty}\biggl|\frac{1}{t^{\rho}\phi(t)^\alpha}
\int_{B(t)}f(x-\phi(|y|)y')\frac{\Omega(y)}{|y|^{n-\rho}}dy\biggr|^q
\frac{dt}{t}\biggr)^{\frac{1}{q}},
\end{equation}
and
\begin{equation}\label{def:fracMarSurf-b-q}
\mu_{\Omega,\rho,\phi,\alpha,q}^{(b)}f(x)
=\biggl(\int_{0}^{\infty}\biggl|\frac{1}{t^{\rho}\phi(t)^\alpha}
\int_{B(t)}f(x-\phi(|y|)y')\frac{b(|y|)\Omega(y)}{|y|^{n-\rho}}dy
\biggr|^q\frac{dt}{t}\biggr)^{\frac{1}{q}}.
\end{equation}
Now we formulate our main theorem.
Here and below we write $\R_+:=(0,\infty)$.
\begin{theorem}\label{thm:fracMarcSurf-1}
Let $\rho>0$, $1<p,\, q<\infty$, and $\Omega\in L^1(\Sn)$. 
Let $c_0>1$ and $c_1>0$.
Suppose that $\phi:\R_+ \to \R_+$ 
is a nonnegative increasing $C^1$-function such that
\begin{align}
\phi(2t)&\le c_0\phi(t)\quad\text{for all}\ \ t\in\mathbb R_+
\label{eq:doubling}
\end{align}
and that
\begin{align}
\phi(t)&\le
c_1t\phi^\prime(t)\quad\text{for all}\ \ t\in\mathbb R_+.\label{eq:uniform}
\end{align}
Define
\[\varphi(t):=\frac{\phi(t)}{t\phi^\prime(t)}
\quad\text{for all}\ \ t\in\mathbb R_+.\]
Then: 
\begin{enumerate}
\item[{\rm(i)}]
Let
\begin{equation}\label{eq:121108-201}
\alpha\in\left(0,\frac{4}{\tilde{p}\tilde{q}c_1\log_2c_0}\right).
\end{equation}
If $\Omega$ satisfies the cancellation condition \eqref{eq:cancellation}, then
\begin{equation}\label{eq:fracMarSurf-0-3}
\|\mu_{\Omega,\rho,\phi,\alpha,q}f\|_{L^p(\Rn)}
\le C 
\|\Omega\|_{L^1(\Sn)}\|f\|_{\dot F_{pq}^{\alpha}(\Rn)}
\end{equation}
for all $f \in {\mathcal S}_\infty(\R^n)$.
\item[{\rm(ii)}]
Let
\[
\alpha\in\left(
-\min\left\{\frac{4\beta}{c_1\log_2c_0 \cdot\tilde{p}\tilde{q}},\,
\frac{\rho}{\log_2c_0}\right\},0\right).
\]
If $\phi$ satisfies the following additional condition
\begin{equation}
\varphi(t) \text{ or }t\phi'(t)\text{ is monotonic on }\R_+\label{eq:monotone},
\end{equation}
and $\Omega$ satisfies
\begin{equation}\label{eq:fracMarSurf-0-4}
Z_\Omega
:=\sup_{\xi'\in \Sn}\int_{\Sn}\frac{|\Omega(y')|}{|y'\cdot\xi'|^\beta}\,
d\sigma(y')<+\infty,
\end{equation}
for some  $0<\beta\le1$,
then 
\begin{equation}\label{eq:121108-19}
\|\mu_{\Omega,\rho,\phi,\alpha,q}f\|_{L^p(\Rn)}
\le C Z_\Omega \|f\|_{\dot F_{pq}^{\alpha}(\Rn)}
\end{equation}
for all $f \in {\mathcal S}_\infty(\R^n)$.
\item[{\rm(iii)}]
Let $\alpha=0$.
If $\Omega\in L\log L(\Sn)$ and it satisfies the cancellation condition
 \eqref{eq:cancellation}, then
\begin{equation}\label{eq:121108-1}
\|\mu_{\Omega,\rho,\phi,\alpha,q}f\|_{L^p(\Rn)}
\le C 
\|\Omega\|_{L\log L(\Sn)}\|f\|_{\dot F_{pq}^{\alpha}(\Rn)}
\end{equation}
for all $f \in {\mathcal S}_\infty(\R^n)$.
\end{enumerate}
In any case,
by density we can extend
$(\ref{eq:fracMarSurf-0-3})$,
$(\ref{eq:121108-19})$
and
$(\ref{eq:121108-1})$
and have them
for all $f \in \dot{F}_{pq}^{\alpha}(\Rn)$.
\end{theorem}
Note that (\ref{eq:doubling}) is referred to as
the doubling condition. Thanks to the useful 
conversation with Professor X.~X. Tau and Miss S.~He in the Zhejiang University
 of Science and Technology,
we could improve our results. 

We state our main result in full generality.
Theorem \ref{thm:fracMarcSurf-1} is almost a direct consequence
of the next theorem;
\begin{theorem}\label{thm:fracMarcSurf-1-b}
Suppose that we are given
$\Omega\in L^1(\Sn)$,
$\phi \in C^1(\R_+,\R_+)$
and parameters
$p,q,\alpha,\gamma,\rho$
satisfying
$$
1<p, q<\infty,
\rho>0,
\gamma>\frac12\max\{\tilde{p},\tilde{q}\},
$$ 
in addition to
\eqref{eq:doubling} 
and \eqref{eq:uniform}
in Theorem {\rm\ref{thm:fracMarcSurf-1}}
Then: 
\begin{enumerate}
\item[{\rm(i)}]
Assume that
\begin{equation}\label{eq:121108-202}
\alpha
\in\left(0,\frac{4}{c_1\log_2c_0}
\cdot
\frac{1/\tilde{p}-1/(2\gamma)}{1-1/\gamma}
\cdot
\frac{1/\tilde{q}-1/(2\gamma)}{1-1/\gamma}
\right).
\end{equation}
If 
$b\in\Delta_\gamma(\R_+)$
and
$\Omega$ satisfies the cancellation condition 
\eqref{eq:cancellation},  
then
\begin{equation}\label{eq:fracMarSurf-0-3-b}
\|\mu_{\Omega,\rho,\phi,\alpha,q}^{(b)}f\|_{L^p(\Rn)}
\le C
\|\Omega\|_{L^1(\Sn)} 
\|b\|_{\Delta_\gamma}
\|f\|_{\dot F_{pq}^{\alpha}(\Rn)}
\end{equation}
for all $f \in {\mathcal S}_\infty(\R^n)$.
\item[{\rm(ii)}]
Assume 
$
\alpha \in\left(
-\min\left\{\frac{2\beta}{c_1\log_2c_0}
\cdot
\frac{1/\tilde{p}-1/(2\gamma)}{1-1/\gamma}
\cdot
\frac{1/\tilde{q}-1/(2\gamma)}{1-1/\gamma},
\frac{\rho}{\log_2c_0}\right\},0
\right)
$
for some $\beta \in (0,1]$.
If
$b\in\Delta_{\max(\gamma,2)}$
and
\begin{equation}\label{eq:fracMarSurf-0-4-b}
W_\Omega
:=
\sup_{\xi'\in \Sn}
\sqrt{
\int_{\Sn\times\Sn}
\frac{|\Omega(y')\Omega(z')|}{|(y'-z')\cdot\xi'|^\beta}
\,d\sigma(y')d\sigma(z')}
<+\infty,
\end{equation}
then 
\begin{equation}\label{eq:121108-29}
\|\mu_{\Omega,\rho,\phi,\alpha,q}^{(b)}f\|_{L^p(\Rn)}
\le C
W_\Omega
\|b\|_{\Delta_{\max(\gamma,2)}}
\|f\|_{\dot F_{pq}^{\alpha}(\Rn)}
\end{equation}
for all $f \in {\mathcal S}_\infty(\R^n)$.
\item[{\rm(iii)}]
Assume $\alpha=0$.
If
$b\in\Delta_{\max(\gamma,2)}$,
$\Omega\in L\log L(\Sn)$ 
and it satisfies the 
cancellation condition \eqref{eq:cancellation}, 
then
\begin{equation}\label{eq:121108-2}
\|\mu_{\Omega,\rho,\phi,\alpha,q}^{(b)}f\|_{L^p(\Rn)}
\le C
\|\Omega\|_{L\log L(\Sn)}
\|b\|_{\Delta_{\max(\gamma,2)}}
\|f\|_{\dot F_{pq}^{\alpha}(\Rn)}
\end{equation}
for all $f \in {\mathcal S}_\infty(\R^n)$.
\end{enumerate}
In any case,
by density we can extend
$(\ref{eq:fracMarSurf-0-3-b})$,
$(\ref{eq:121108-29})$
and
$(\ref{eq:121108-2})$
and have them
for all $f \in \dot{F}_{pq}^{\alpha}(\Rn)$.
\end{theorem}

Theorem \ref{thm:fracMarcSurf-1}(i) and (iii)
are direct consequences of Theorem \ref{thm:fracMarcSurf-1-b}.
Indeed, assuming (\ref{eq:121108-201})
and choosing $\gamma \gg 1$,
we have (\ref{eq:121108-202}).
So, to obtain (i)
we can apply Theorem \ref{thm:fracMarcSurf-1-b}
for such $\gamma$
with $b \equiv 1$.
Theorem \ref{thm:fracMarcSurf-1}(iii)
is a direct conseuqence of 
Theorem \ref{thm:fracMarcSurf-1-b}(iii).
Note that in Theorems \ref{thm:fracMarcSurf-1}(ii)
and \ref{thm:fracMarcSurf-1-b}(ii),
the conditions of $\alpha$ is slightly improved.


We rely upon the modified Littlewood-Paley decomposition
for the proof of Theorem \ref{thm:fracMarcSurf-1-b},
which we shall describe now.
Let $\{a_k\}_{k\in\Z}$ be a lacunary sequence 
of positive numbers in the sense that
$a_{k+1}/a_{k}\ge a>1$ $(k\in\Z)$. 
A sequence $\{\Phi_k\}_{k\in\Z}$ of 
$C^\infty(\Rn)$-functions is said to be a partition of unity 
adapted to $\{a_k\}_{k\in\Z}$ if 
$$\supp \widehat{\Phi_k}\subset \{\xi\in\Rn;\,a_{k-1}\le
 |\xi|\le a_{k+1}\} \quad (k\in\Z),$$ 
$$\sum_{k\in\Z}\widehat{\Phi_k}(\xi)=1 \quad (\xi\in\Rn\setminus \{0\}),
$$ 
and 
$$
|\xi^\beta\partial^\beta \widehat{\Phi_k}(\xi)|\le C_\beta
$$ for any 
multiindex $\beta$.  

Denote by ${\mathcal P}$ the set of all polynomials
in ${\mathbb R}^n$.
Let $1< p,q<\infty$ and $\alpha\in\R$. For $f\in\S'(\Rn)/{\mathcal P}$, we define 
the norm $\|f\|_{\dot F_{pq}^{\alpha,\{\Phi_k\}_{k\in\Z}}(\Rn)}$ by 
\begin{equation}\label{eq:Triebel-Liz1}
\|f\|_{\dot F_{pq}^{\alpha,\{\Phi_k\}_{k\in\Z}}(\Rn)}
=\Bigl\|\Bigl(\sum_{k\in\Z}a_k^{\alpha q}|\Phi_k*f|^q\Bigr)^{1/q}
\Bigr\|_{L^p(\Rn)}.
\end{equation}

We admit that Proposition \ref{lem:Triebel-Liz1} below
is true and we prove
Theorem \ref{thm:fracMarcSurf-1-b} first.
We postpone the proof of Proposition \ref{lem:Triebel-Liz1}
until the end of the paper.
\begin{proposition}\label{lem:Triebel-Liz1}
Let  $\alpha\ne0$ and $1<p,q<\infty$. 
Let $\{a_k\}_{k\in\Z}$ be a lacunary sequence of positive numbers 
with $a_{k+1}/a_{k}\ge a>1$ $(k\in\Z)$. 
If $\|f\|_{\dot F_{pq}^{\alpha,\{\Phi_k\}_{k\in\Z}}(\Rn)}$ and 
$\|f\|_{\dot F_{pq}^{\alpha,\{\Psi_k\}_{k\in\Z}}(\Rn)}$ 
are equivalent for any two partitions of unity, 
$\{\Phi_k\}_{k\in\Z}$ and $\{\Psi_k\}_{k\in\Z}$, adapted 
to $\{a_k\}_{k\in\Z}$, 
then there exists $C_0>a$ such that 
\begin{equation*}
\frac{a_{k+1}}{a_k}\le C_0\qquad(k\in\Z),
\end{equation*}
and, in this case, 
$\|f\|_{\dot F_{pq}^{\alpha,\{\Phi_k\}_{k\in\Z}}(\Rn)}$ 
is equivalent to the usual homogeneous Triebel-Lizorkin space norm 
$\|f\|_{\dot F_{pq}^{\alpha}(\Rn)}$. 
\end{proposition}

In Sections \ref{section:3}--\ref{section:5}, 
we shall prove Theorems 
\ref{thm:fracMarcSurf-1} 
and
\ref{thm:fracMarcSurf-1-b}
as well as Proposition \ref{lem:Triebel-Liz1},
respectively.


\section{A strategy of the proof of Theorem \ref{thm:fracMarcSurf-1-b}}
\label{section:2}

\subsection{A setup}

For $t>0$,
a function $b$ on $\R_+$
and a homogeneous kernel $\Omega$ on $\Rn$, 
assume
\[
\int_{B(t) \setminus B(t/2)}|b(|x|)\Omega(x')|\,dx
<\infty.
\]
For $\rho>0$ and a nice function,
we define the family $\{\sigma_{t};\,t\in \mathbb R_+\}$ of measures 
and the maximal operator $\sigma^*$ on $\mathbb R^n$ by
\begin{align}
\int_{\Rn}f(x)\,d\sigma_{t}(x)&
=\frac{1}{t^\rho}\int_{B(t) \setminus B(t/2)}
f\bigl(\phi(|x|)x')\frac{b(|x|)\Omega(x')}{|x|^{n-\rho}}\,dx,
\label{eq:mu-OmegaSurf-2-2}
\\
\sigma^* f(x)&=\sup_{t>0}\bigl||\sigma_t|*f(x)\bigr|
\quad(x \in \Rn).
\label{eq:sigma-max}
\end{align}
Note that the mapping
$
x \in {\mathbb R}^n \setminus \{0\} 
\mapsto \phi(|x|)x' \in {\mathbb R}^n \setminus \overline{B(\inf \phi)} 
$
is a $C^1$-diffeomorphism,
since $\phi \in C^1(\R_+,\R_+)$ 
satisfies \eqref{eq:doubling} and \eqref{eq:uniform}.
Therefore,
if we consider the measure $\sigma^\dagger_t$
by
\[
\int_{\Rn}f(x)\,d\sigma^\dagger_{t}(x)
=\frac{1}{t^\rho}\int_{B(t) \setminus B(t/2)}
f\bigl(x)\frac{b(|x|)\Omega(x')}{|x|^{n-\rho}}\,dx,
\]
then the above diffeomorphism induces $\sigma_t$.
So, about the absolute value of $\sigma_t$,
we have
\[
\int_{\Rn}f(x)\,d|\sigma_{t}|(x)
=\frac{1}{t^\rho}\int_{B(t) \setminus B(t/2)}
f\bigl(\phi(|x|)x')\frac{|b(|x|)\Omega(x')|}{|x|^{n-\rho}}\,dx.
\]
A direct consequence of this alternative definition
of $|\sigma_{t}|$ is that
we have
\begin{equation}\label{eq:121130-1}
\|\sigma_t\|
\le C
\|b\|_{\Delta_1}
\|\Omega\|_{L^1}.
\end{equation}
If we use (\ref{eq:mu-OmegaSurf-2-2}),
then we can write
\begin{equation}\label{eq:mu-OmegaSurf-3-2}
\tilde\mu_{\Omega,\alpha,\rho,q}^{(b)}(f)(x)
=\biggl(\int_{0}^{\infty}|\sigma_t*f(x)|^q\frac{dt}{t\phi(t)^{q\alpha}}
\biggr)^{1/q}\quad(x \in \Rn).
\end{equation}

\begin{lemma}\label{lem:121108-0}
Let $\Omega \in L^1(\Sn)$.
\begin{enumerate}
\item
For all admissible parameters,
\begin{align}\label{eq:FTsigma-1}
|\widehat {\sigma_t}(\xi)|&\le
2^{n-\rho}\|\Omega\|_{L^1(\Sn)}\|b\|_{\Delta_1}
\quad(t>0, \, \xi\in\Rn).
\end{align}
\item
If in addition $\Omega$ satisfies
$(\ref{eq:cancellation})$,
then we have
\begin{equation}\label{eq:FTsigma-canc-1}
|\widehat {\sigma_t}(\xi)|
\le 2\|\Omega\|_{L^1(\Sn)}\|b\|_{\Delta_1}\phi(t)|\xi|
\quad(t>0, \, \xi\in\Rn).
\end{equation}
\end{enumerate}
\end{lemma}

\begin{proof}
\quad
\begin{enumerate}
\item
From the definition of the Fourier transform,
we have an expression of $\widehat {\sigma_t}(\xi)$;
\begin{align}
\widehat {\sigma_t}(\xi)&=\frac{1}{t^\rho}\int_{B(t) \setminus B(t/2)}
\frac{b(|y|)\Omega(y')}{|y|^{n-\rho}}e^{-i\phi(|y|)y'\cdot \xi}dy.
\label{eq:FTsigma}
\end{align}
From (\ref{eq:FTsigma}) we get (\ref{eq:FTsigma-1}).
\item
Using the cancellation property (\ref{eq:cancellation})
of $\Omega$, we have another expression 
of $\widehat {\sigma_t}(\xi)$;
\begin{align}
\widehat {\sigma_t}(\xi)
&=\frac{1}{t^\rho}\int_{B(t)\setminus B(t/2)}
\frac{b(|y|)\Omega(y')}{|y|^{n-\rho}}
\bigl(e^{-i\phi(|y|)y'\cdot \xi}-1\bigr)\,dy.\label{eq:FTsigma-canc}
\end{align}
From the monotonicity of $\phi$,
(\ref{eq:doubling}) and (\ref{eq:FTsigma-canc}) we obtain
\begin{align*}
|\widehat {\sigma_t}(\xi)|
&\le\frac{1}{t^\rho}\int_{t/2}^{t}
\left(\int_{\Sn}|\Omega(y')|d\sigma(y')\right)\,
|\xi|\cdot|\phi(r)b(r)|r^{\rho-1}dr
\\
&\le \|\Omega\|_{L^1(\Sn)}\phi(t){|\xi|}\int_{t/2}^{t}|b(r)|\frac{dr}{r}\notag
\\
&\le 2\|\Omega\|_{L^1(\Sn)}\|b\|_{\Delta_1}\phi(t)|\xi|.
\notag
\end{align*}
So we are done.
\end{enumerate}
\end{proof}
\par

As for the maximal operator $\sigma^*$ given by (\ref{eq:sigma-max}), 
we invoke the following lemma
in \cite[Lemma 3.2]{DingXueY1}:
We define the directional Hardy-Littlewood maximal function
of $F$ for a fixed vector $y' \in S^{n-1}$ by
\[
M_{y'}F(x)
=
\sup_{r>0}\frac{1}{2r}\int_{-r}^r |f(x-ty')|\,dt.
\]
By the orthogonal decomposition
$\Rn=H \oplus \R y'$,
we can prove that $M_{y'}$ is bounded on $L^p(\Rn)$
for all $1<p<\infty$ and that the bound 
is uniform over $y'$.
By combining the H\"{o}lder inequality
and the change of variables to polar coordinates,
we can prove;
\begin{lemma} \label{lem:121030-2}
Let $\gamma>1$.
Then there exists $C>0$ such that
\begin{equation}\label{eq:sigma-max-HL}
\sigma^*(f)(x)
\le C\|b\|_{\Delta_\gamma}\|\Omega\|_{L^1(S^{n-1})}^{1/\gamma}
\biggl(\int_{S^{n-1}}|\Omega(y')|M_{y'}(|f|^{\gamma'})(x)\,d\sigma(y')
\biggr)^{1/{\gamma'}}
\end{equation}
for all $x \in \Rn$.
\end{lemma}
Thanks to Lemma \ref{lem:121030-2} and the Minkowski inequality, 
for $p>\gamma'$ there exists $C>0$ such that
\begin{equation}\label{eq:sigma-maxLp}
\|\sigma^*(f)\|_{L^p(\Rn)}
\le C\|b\|_{\Delta_\gamma}\|\Omega\|_{L^1(S^{n-1})}\|f\|_{L^p(\Rn)}.
\end{equation}

From the monotonicity,
\eqref{eq:doubling} and \eqref{eq:FTsigma-canc-1} we get, for 
$\alpha\in\R,\,k\in\Z$,
\begin{equation}\label{eq:mu-j-7-6} 
\biggl(\int_{2^k}^{2^{k+1}}|\widehat {\sigma_t}(\xi)|^2
\frac{dt}{t\phi(t)^{2\alpha}}\biggr)^{1/2}
\le 2\|\Omega\|_{L^1(\Sn)}\|b\|_{\Delta_1}|\xi|
\frac{\phi(2^{k})}{\phi(2^k)^{\alpha}}.
\end{equation}
Using \eqref{eq:doubling} and \eqref{eq:uniform}, 
we have;
\begin{lemma} \label{lem:121030-20}
For any $0\leq\beta< 1$,
\begin{equation}\label{eq:mu-j-7-3-1-b}
|\widehat {\sigma_t}(\xi)|
\le C
W_\Omega\|b\|_{\Delta_2}
\frac1{(|\xi|\phi(t))^{\beta/2}}
\end{equation}
for $\xi \in \Rn$.
\end{lemma}

For a precise proof, 
see the proof of \cite[Lemma 2.4]{DingXueY1}.

\subsection{Properties of $\phi$}

We denote $a_j:=1/\phi(2^{-j})$ and 
$a:=2^{1/\|\varphi\|_{L^\infty(\R_+)}}>1$. 
Then $\{a_j\}_{j\in\Z}$ is 
also a lacunary sequence of the same lacunarity 
as $\{\phi(2^j)\}_{j\in\Z}$. 
From the assumption (\ref{eq:doubling}), it follows that 
\begin{equation}\label{eq:121105-1000}
\phi(2^kt)\le c_0^k\phi(t)
\end{equation}
for $k\in\N$. 
It is easily seen from (\ref{eq:uniform})
that $\{\phi(2^j)\}_{j\in\Z}$ is a lacunary sequence of 
positive numbers satisfying 
\begin{equation}\label{eq:121105-1}
\phi(2^kt)=
\phi(t)
\int_{t}^{2^kt}\left(\log \phi(s)\right)'\,ds
\ge 2^{k/\|\varphi\|_{L^\infty({\mathbb R}_+)}}\phi(t)
=a^k\phi(t) \quad (t>0)
\end{equation}
for $k\in\N$. 
See e.g. \cite[Lemma 2.8]{DingXueY1}
for details.

Note also that,
for $\phi \in C^1(\R_+,\R_+)$ satisfying (\ref{eq:doubling}),
the condition (\ref{eq:uniform}) implies
\begin{equation}\label{eq:uniform-2}
\phi(2t) \ge C_1\phi(t)
\quad(t>0)
\end{equation}
for some $C_1>1$.
Indeed, assuming (\ref{eq:doubling}),
there exists $s \in[t,2t]$
\[
\phi(2t)-\phi(t)
=
t\phi'(s)
\ge
c_1\frac{t}{s}\phi(s) 
\ge
\frac{c_1}{c_0}\phi(t)
\]
by the mean value theorem,
proving (\ref{eq:uniform-2}).

If in addition $\phi$ is concave,
then (\ref{eq:uniform-2}) implies (\ref{eq:uniform}).
Indeed,
\[
\phi'(2t) \ge \frac{\phi(2t)-\phi(t)}{t}
\ge (C_1-1)\frac{\phi(t)}{t} \quad (t>0).
\]
\subsection{Construction of partition of unity}

For our purpose, we introduce a partition of unity and a characterization of 
the homogeneous Triebel-Lizorkin spaces associated to $\phi$
satisfying \eqref{eq:doubling} and \eqref{eq:uniform}. 
 
Take a nonincreasing $C^\infty\bigl(\R)$-function $\eta$ 
such that $\chi_{[-1/a,1/a]}(t) \le \eta(t) \le \chi_{[-1,1]}(t)$
for all $t \in \R$.\\
We define functions $\psi_j$ on $\Rn$ by
\begin{equation}\label{eq:TL-unit-decomp0}
\psi_j(\xi)
=\eta\Bigl(\frac{|\xi|}{a_{j+1}}\Bigr)
-\eta\Bigl(\frac{|\xi|}{a_{j}}\Bigr)
\quad(\xi \in \Rn).
\end{equation}
Then observe that
\begin{equation}\label{eq:TL-unit-decomp1}
\psi_j(\xi)=\begin{cases}
0, &0\le |\xi|\le a_j/a,\ |\xi|\ge aa_{j+1}, \\
1, & a\,a_j\le t\le a_{j+1},
\end{cases}
\end{equation}
and that
\begin{align}
&\supp \psi_j\subset \{a_{j}/a\le|\xi|\le aa_{j+1}\},
\label{eq:TL-unit-decomp2}
\\
\quad
&\supp \psi_j\cap\supp \psi_\ell=\emptyset,\text{ for }|j-\ell|\ge2.
\label{eq:TL-unit-decomp3}
\\
&\sum_{j\in\Z}\psi_j(\xi)=1\quad(\Rn\setminus \{0\}).
\label{eq:TL-unit-decomp4}
\end{align}
That is, $\{\psi_j\}_{j \in \Z}$ 
is a smooth partition of unity adapted to 
$\{a_j\}_{j \in \Z}$. 

Let $\Psi_j$ be defined on $\Rn$ by $\widehat{\Psi_j}(\xi)=\psi_j(\xi)$
for $\xi \in \Rn$.
By Proposition \ref{lem:Triebel-Liz1},
we have
\begin{equation}\label{eq:TL-unit-decomp5}
\biggl\|\biggl(\sum_{j=-\infty}^\infty|a_j^{\alpha}\Psi_j*f|^q\biggr)^{1/q}
\biggr\|_{L^p}\approx \|f\|_{\dot F_{pq}^\alpha(\Rn)}.
\end{equation}
if $a_{j+1}/a_j\le b$ $(j\in\Z)$ for some $b\ge a$. 

This condition is satisfied in our case, i.e. 
$a_{j+1}/a_j=\phi(2^{-j})/\phi(2^{-j-1})\le c_1$. \par

\subsection{A reduction by using the scaling invariance}

Now, using the definition of $\mu_{\Omega,\rho,\phi,\alpha,q}^{(b)}(f)(x)$ and 
the triangle inequality, via change of variables
$y \mapsto 2^ky$, we obtain 
\begin{align*}
&\mu_{\Omega,\rho,\phi,\alpha,q}^{(b)}(f)(x)
=\biggl(\int_{0}^{\infty}\biggl|\frac{1}{t^\rho\phi(t)^\alpha}\int_{B(t)}
\frac{b(|y|)\Omega(y')}{|y|^{n-\rho}}f\bigl(x-\phi(|y|)y'\bigr)\,dy\biggr|^q
\frac{dt}{t}\biggr)^{1/q}
\\
&=\biggl(\int_{0}^{\infty}\biggl|
\sum_{k=0}^{\infty}\frac{1}{t^\rho\phi(t)^\alpha}
\int_{B(2^{-k}t) \setminus B(2^{-k-1}t)}
\frac{b(|y|)\Omega(y')}{|y|^{n-\rho}}f\bigl(x-\phi(|y|)y'\bigr)\,dy\biggr|^q
\frac{dt}{t}\biggr)^{1/q}
\\
&\le\sum_{k=0}^{\infty}\biggl(
\int_{0}^{\infty}\biggl|\frac{1}{t^\rho\phi(t)^\alpha}
\int_{B(2^{-k}t) \setminus B(2^{-k-1}t)}
\frac{b(|y|)\Omega(y')}{|y|^{n-\rho}}f\bigl(x-\phi(|y|)y'\bigr)\,dy\biggr|^q
\frac{dt}{t}\biggr)^{1/q}
\\
&=\sum_{k=0}^{\infty}\frac{1}{2^{\rho k}}\biggl(
\int_{0}^{\infty}\biggl|\frac{1}{t^\rho\phi(2^kt)^\alpha}
\int_{B(t)\setminus B(t/2)}
\frac{b(2^k|y|)\Omega(y')}{|y|^{n-\rho}}f\bigl(x-\phi(2^k|y|)y'\bigr)\,dy
\biggr|^q
\frac{dt}{t}\biggr)^{1/q}.
\end{align*}
Hence
\begin{align}
\lefteqn{
\label{eq:121030-1}\quad
\mu_{\Omega,\rho,\phi,\alpha,q}^{(b)}(f)(x)
}\\ 
\nonumber
&
\le\sum_{k=0}^{\infty}\frac{1}{2^{\rho k}}
\biggl(
\int_{0}^{\infty}\biggl|\frac{1}{t^\rho\phi(2^kt)^\alpha}
\int_{B(t)\setminus B(t/2)}
\frac{b(2^k|y|)\Omega(y')}{|y|^{n-\rho}}
f\bigl(x-\phi(2^k|y|)y'\bigr)\,dy\biggr|^q
\frac{dt}{t}\biggr)^{1/q}.
\end{align}
So, in the case $\alpha\ge0$ we have 
\begin{align*}
&\mu_{\Omega,\rho,\phi,\alpha,q}^{(b)}(f)(x)
\le\sum_{k=0}^{\infty}\frac{1}{2^{(\rho+\alpha/\|\varphi\|_\infty)k}}
\\
&\quad \quad\times\biggl(
\int_{0}^{\infty}\biggl|\frac{1}{t^\rho\phi(t)^\alpha}
\int_{B(t)\setminus B(t/2)}
\frac{b(2^k|y|)\Omega(y')}{|y|^{n-\rho}}
f\bigl(x-\phi(2^k|y|)y'\bigr)\,dy\biggr|^q
\frac{dt}{t}\biggr)^{1/q}.
\end{align*}
So, in the case $0>\alpha>-\rho/\log c_0$, from (\ref{eq:121030-1}), we have 
\begin{align*}
&\mu_{\Omega,\rho,\phi,\alpha,q}^{(b)}(f)(x)
\le\sum_{k=0}^{\infty}\frac{1}{2^{(\rho+\alpha\log_2 c_0)k}}
\\
&\quad \quad\times\biggl(
\int_{0}^{\infty}\biggl|\frac{1}{t^\rho\phi(t)^\alpha}
\int_{B(t)\setminus B(t/2)}
\frac{b(2^k|y|)\Omega(y')}{|y|^{n-\rho}}f
\bigl(x-\phi(2^k|y|)y'\bigr)\,dy\biggr|^q\frac{dt}{t}\biggr)^{1/q}.
\end{align*}
Notice that
$b$ and $b(2^k\cdot)$
satisfy the same condition
due to the scaling invariance
of $\Delta_\gamma$.
Likewise
$\phi$ and $\phi(2^k\cdot)$
satisfy the same conditions
\eqref{eq:doubling} and \eqref{eq:uniform}
with constants independent of $k$.
Hence, for our purpose, it is sufficient to consider
the modified operator given by
\[
\tilde\mu_{\Omega,\rho,\phi,\alpha,q}^{(b)}(f)(x)\\
:=\biggl(\int_{0}^{\infty}\biggl|\frac{1}{t^\rho\phi(t)^\alpha}
\int_{B(t) \setminus B(t/2)}
\frac{b(|y|)\Omega(y')}{|y|^{n-\rho}}
f\bigl(x-\phi(|y|)y'\bigr)\,dy\biggr|^q
\frac{dt}{t}\biggr)^{1/q}
\]
for $x \in \Rn$.



Now we proceed to the proof of Theorem \ref{thm:fracMarcSurf-1-b}. 
Let
\begin{equation}\label{eq:fracMarSurf-q-2}
\tilde\mu_{\Omega,\rho,\phi,\alpha,q,j}^{(b)}f(x)
:=\biggl(\sum_{k=-\infty}^{\infty}
\int_{2^k}^{2^{k+1}}|\Psi_{j-k}*\sigma_t*f(x)|^q\frac{dt}
{t\phi(t)^{q\alpha}}\biggr)^{1/q}
\quad(x\in\Rn)
\end{equation}
for each $j$.
Using the partition of unity \eqref{eq:TL-unit-decomp0}
and the triangle inequality, we then have
\begin{align}\label{eq:fracMarSurf-q-1}
\tilde\mu_{\Omega,\rho,\phi,\alpha,q}^{(b)}f(x)
&=\biggl(\int_{0}^{\infty}\Bigl|\sum_{j\in\Z}\Psi_{j}*\sigma_t*f(x)\Bigr|^q
\frac{dt}{t\phi(t)^{q\alpha}}\biggr)^{1/q}
\notag\\
&=\biggl(\sum_{k\in\Z}\int_{2^k}^{2^{k+1}}
\Bigl|\sum_{j\in\Z}\Psi_{j-k}*\sigma_t*f(x)\Bigr|^q\frac{dt}
{t\phi(t)^{q\alpha}}\biggr)^{1/q}
\notag\\
&\le \sum_{j\in\Z}\biggl(\sum_{k\in\Z}\int_{2^k}^{2^{k+1}}
\Bigl|\Psi_{j-k}*\sigma_t*f(x)\Bigr|^q\frac{dt}{t\phi(t)^{q\alpha}}
\biggr)^{1/q}
\notag\\
&\le \sum_{j\in\Z}\tilde\mu_{\Omega,\rho,\phi,\alpha,q,j}^{(b)}f(x).
\end{align}

Next, we treat the $L^p$-estimate of 
$\tilde\mu_{\Omega,\rho,\phi,\alpha,q,j}^{(b)}f$.

Let us set
\[
\alpha(j)
:=
\begin{cases}
\alpha/c_1,&j \ge 0,\\
\alpha\log_2c_0,&j<0.
\end{cases}
\]
In Section \ref{section:4}
we plan to distinguish three cases to prove;
\begin{lemma}\label{lem:121108-1}
Assume either one of the following three conditions{\rm;}
\begin{enumerate}
\item
$1<q<r<\gamma q<\infty$.
\item
$1<q'<r'<\gamma q'<\infty$.
\item
$1<q=r<\infty$.
\end{enumerate}
If $\Omega \in L^1(\Sn)$, then we have
\begin{equation}\label{eq:mu-j-Lp-5}
\|\tilde\mu_{\Omega,\rho,\phi,\alpha,q,j}^{(b)}f\|_{L^r(\Rn)}
\le C
2^{-\alpha(j)j}\|\Omega\|_{L^1(\Sn)}\|f\|_{\dot F_{rq}^{\alpha}(\Rn)},
\end{equation}
\end{lemma}

However, in Case 3,
we just interpolate Cases 1 and 2.
So we concentrate on Cases 1 and 2 in Section \ref{section:4}.

Note that Cases 1--3 do not cover all the cases
as the above images show.

We also need to prove;
\begin{lemma}\label{lem:121108-2}
Let $\phi$ satisfy the same conditions
\eqref{eq:doubling} and \eqref{eq:uniform}.
Assume that $\Omega \in L^1(\Sn)$ satisfies the cacellation condition 
$(\ref{eq:cancellation})$.
Then
\begin{equation}\label{eq:mu-j-7-9}
\|\tilde\mu_{\Omega,\rho,\phi,\alpha,2,j}^{(b)}f\|_{L^2(\Rn)}
\le C
2^{(\alpha(-j)/\alpha-\alpha(j))j}
\|\Omega\|_{L^1(\Sn)}
\|b\|_{\Delta_\gamma}
\|f\|_{\dot F_{2,2}^\alpha}.
\end{equation}
\end{lemma}

By using the strong decay
of (\ref{eq:mu-j-Lp-5}),
interpolate (\ref{eq:mu-j-Lp-5}) and (\ref{eq:mu-j-7-9})
to have (\ref{eq:mu-j-Lp-5}) again
for any admissible $p$ and $q$.
Thus, in conclusion,
(\ref{eq:fracMarSurf-q-1}) is summable
over $j$ by virtue of (\ref{eq:mu-j-Lp-5}).

\section{Proof of Theorem \ref{thm:fracMarcSurf-1-b}} 
\label{section:3}

In this section, we prove Theorem \ref{thm:fracMarcSurf-1-b}. One can obtain
Theorem \ref{thm:fracMarcSurf-1-b} by observing carefully the proof of 
\cite[Theorem 6]{ChenFanYing}, but for the sake of completeness, we shall 
give its detailed proof, modifying their one.
\subsection{Proof of Lemma \ref{lem:121108-1}}

Here we do not need the cancellation property
of $\Omega$
and hence we can consider its absolute value
of $\sigma_t$.

\begin{enumerate}
\item
In case $q<r<\gamma q$, let
\begin{equation*}
J:=\|\tilde\mu_{\Omega,\rho,\phi,\alpha,q,j}^{(b)}f\|_{L^r(\Rn)}^q
=\biggl\|\biggl(\sum_{k\in\Z}\int_{2^k}^{2^{k+1}}
\Bigl|\Psi_{j-k}*\sigma_t*f\Bigr|^q\frac{dt}{t\phi(t)^{\alpha q}}
\biggr)^{1/q}\biggr\|_{L^r(\Rn)}^q.
\end{equation*}
Let us set $s=(r/q)'=r/(r-q)$. By the duality $L^{q/r}$-$L^s$, 
we can take a nonnegative function
$h\in L^s(\Rn)$ with $\|h\|_{L^s(\Rn)}=1$ such that
\begin{align*}
J&=\int_{\Rn}\sum_{k\in\Z}\left\{\int_{2^k}^{2^{k+1}}
\Bigl|\Psi_{j-k}*\sigma_t*f(x)\Bigr|^q\frac{dt}{t\phi(t)^{\alpha q}}
\right\}\,h(x)\,dx.
\end{align*}
Denote by $\|\sigma_t\|$ the total mass of $\sigma_t$.
Hence,
by the H\"{o}lder inequality 
\begin{align*}
J&=\sum_{k\in\Z}\int_{2^k}^{2^{k+1}}\left\{\int_{\Rn}
\Bigl|\int_{\Rn}\Psi_{j-k}*f(x-y)d\sigma_t(y)\Bigr|^q\,h(x)\,dx
\right\}
\frac{dt}{t\phi(t)^{\alpha q}}
\\
&\le \sum_{k\in\Z}\int_{2^k}^{2^{k+1}}\left\{\int_{\Rn}\left[
\int_{\Rn}|\Psi_{j-k}*f(x-y)|^qd|\sigma_t|(y)\right]\|\sigma_t\|^{q/q'}
\,h(x)\,dx\right\}\frac{dt}{t\phi(t)^{\alpha q}}.
\end{align*}
By virtue of (\ref{eq:121130-1}),
we have
\begin{align*}
J
&\le C\|\Omega\|_{L^1(S^{n-1})}^q
\|b\|_{\Delta_1}\\
&\quad \times\sum_{k\in\Z}
\int_{2^k}^{2^{k+1}}
\left\{\int_{\Rn}\left[\int_{\Rn}|\Psi_{j-k}*f(y)|^qd|\sigma_t|(x-y)\right]
\,h(x)\,dx\right\}\frac{dt}{t\phi(t)^{\alpha q}}
\\
&= C\|\Omega\|_{L^1(S^{n-1})}^q
\|b\|_{\Delta_1}\\
&\quad \times \sum_{k\in\Z}
\int_{2^k}^{2^{k+1}}\left\{\int_{\Rn}|\Psi_{j-k}*f(y)|^q
\biggl(\int_{\Rn}h(x)\,d|\sigma_t|(x-y)\biggr)dy\right\}
\frac{dt}{t\phi(t)^{\alpha q}}.
\end{align*} 
Since $1< q<r<\gamma q$, we have $s>\gamma'$.
So, by \eqref{eq:sigma-maxLp} and H\"older's inequality, 
we conclude
\begin{align*}
\lefteqn{
J^{1/q}
}\\
&\le C\|\Omega\|_{L^1(\Sn)}\|b\|_{\Delta_\gamma}
\biggl(\int_{\Rn}\left(\sum_{k\in\Z}
\int_{2^k}^{2^{k+1}}|\Psi_{j-k}*f(y)|^{q}
\sigma^*(h)(y)\frac{dt}{t\phi(t)^{\alpha q}}
\right)dy\biggr)^{1/q}
\\
&\le C\|\Omega\|_{L^1(\Sn)}
\|b\|_{\Delta_\gamma}\biggl(\int_{\Rn}\sum_{k\in\Z}
\frac1{\phi(2^k)^{\alpha q}}|\Psi_{j-k}*f(y)|^{q}\sigma^*(h)(y)dy\biggr)^{1/q}
\\
&\le C\|\Omega\|_{L^1(\Sn)}
\|b\|_{\Delta_\gamma}
\biggl(\int_{\Rn}\biggl(\sum_{k\in\Z}
\frac1{\phi(2^k)^{\alpha q}}
|\Psi_{j-k}*f(y)|^{q}\biggr)^{s'}dy\biggr)^{1/(s'q)}\|h\|_{L^s(\Rn)}^{1/q}
\\
&= C\|\Omega\|_{L^1(\Sn)}\|b\|_{\Delta_\gamma}
\biggl(\int_{\Rn}\biggl(
\sum_{\ell\in\Z}\frac1{\phi(2^{j-\ell})^{\alpha q}}|\Psi_{\ell}*f(y)|^{q}
\biggr)^{r/q}dy\biggr)^{1/r}.
\end{align*}
Thus, we have
\begin{equation}\label{eq:mu-j-Lp-3}
\|\tilde\mu_{\Omega,\rho,\phi,\alpha,q,j}^{(b)}f\|_{L^r(\Rn)}\le 
 C2^{-\alpha(j)j}\|b\|_{\Delta_\gamma}\|f\|_{\dot F_{rq}^{\alpha}}.
\end{equation}
\item
In case $1<r<q$ and $r'<\gamma q'$, 
it follows that $r'>q'$. By duality, there is a sequence of
functions $g_k(x,t)$ such that
\begin{equation*}
\biggl(\int_\Rn\biggl(\sum_{k\in\Z}\int_{2^k}^{2^{k+1}}|g_k(x,t)|^{q'}
\frac{dt}{t}\biggr)^{r'/q'}dx\biggr)^{1/r'}=1,
\end{equation*}
and that
\begin{align*}
&\biggl\|\biggl(\sum_{k\in\Z}\int_{2^k}^{2^{k+1}}
\Bigl|\Psi_{j-k}*\sigma_t*f\Bigr|^q\frac{dt}{t\phi(t)^{\alpha q}}
\biggr)^{1/q}\biggr\|_{L^r(\Rn)}
\\
&=\int_\Rn \sum_{k\in\Z}\left\{\int_{2^k}^{2^{k+1}}
\bigl(\Psi_{j-k}*\sigma_t*f(x)\bigr)g_k(x,t)\frac{dt}{t\phi(t)^{\alpha}}
\right\}dx.
\end{align*}
Then we have 
\begin{align*}
&\int_\Rn \sum_{k\in\Z}\int_{2^k}^{2^{k+1}}
\bigl(\Psi_{j-k}*\sigma_t*f(x)\bigr)g_k(x,t)
\frac{dt}{t\phi(t)^{\alpha}}dx
\\
&\le \int_\Rn \sum_{k\in\Z}\left\{\int_{2^k}^{2^{k+1}}\biggl(\int_\Rn
|\Psi_{j-k}*f(y)|d|\sigma_t|(x-y)\biggr)|g_k(x,t)|
\frac{dt}{t\phi(t)^{\alpha}}\right\}dx
\\
&\le \int_\Rn \sum_{k\in\Z}\left\{\int_{2^k}^{2^{k+1}}|\Psi_{j-k}*f(y)|
\biggl(\int_\Rn|g_k(x,t)|d|\sigma_t|(x-y)\biggr)\frac{dt}{t\phi(t)^{\alpha}}
\right\}dx
\\
&\le C\int_\Rn \sum_{k\in\Z}
\left\{\int_{2^k}^{2^{k+1}}\frac1{\phi(2^k)^{\alpha}}
|\Psi_{j-k}*f(y)|\biggl(\int_\Rn|g_k(x,t)|d|\sigma_t|(x-y)\biggr)\frac{dt}{t}
\right\}dy.
\end{align*}
By using the H\"{o}lder inequality for sequences,
we have
\begin{align*}
&\int_\Rn \left(\sum_{k\in\Z}\int_{2^k}^{2^{k+1}}
\bigl(\Psi_{j-k}*\sigma_t*f(x)\bigr)g_k(x,t)
\frac{dt}{t\phi(t)^{\alpha}}\right)dx
\\
&\le C\int_\Rn \Bigl(\sum_{k\in\Z}\frac1{\phi(2^k)^{\alpha q}}
|\Psi_{j-k}*f(y)|^q\Bigr)^{1/q}
\\
&\quad\times\biggl(\sum_{k\in\Z}
\int_{2^k}^{2^{k+1}}\biggl(\int_\Rn|g_k(x,t)|d|\sigma_t|(x-y)\biggr)^{q'}
\frac{dt}{t}\biggr)^{1/q'}dy.
\end{align*}
By the properties of $\phi$ and
Proposition \ref{lem:Triebel-Liz1},
we conclude
\begin{align*}
&\int_\Rn \left(\sum_{k\in\Z}\int_{2^k}^{2^{k+1}}
\bigl(\Psi_{j-k}*\sigma_t*f(x)\bigr)g_k(x,t)
\frac{dt}{t\phi(t)^{\alpha}}\right)dx
\\
&\le C2^{-\alpha(j) j}\|b\|_{\Delta_\gamma}
\biggl(\int_\Rn \Bigl(\sum_{k\in\Z}
\frac1{\phi(2^k)^{\alpha q}}
|\Psi_{k}*f(y)|^q\Bigr)^{r/q}dy\biggr)^{1/r}
\\
&\quad\times\biggl(\int_\Rn \biggl(\sum_{k\in\Z}
\int_{2^k}^{2^{k+1}}\biggl(\int_\Rn|g_k(x,t)|d|\sigma_t|(x-y)\biggr)^{q'}
\frac{dt}{t}\biggr)^{r'/q'}dy\biggr)^{1/r'}
\\
&= C2^{-\alpha(j) j}\|b\|_{\Delta_\gamma}\|f\|_{\dot F_{rq}^{\alpha}}
\\
&\quad\times\biggl(\int_\Rn \biggl(\sum_{k\in\Z}
\int_{2^k}^{2^{k+1}}\biggl(\int_\Rn|g_k(x,t)|d|\sigma_t|(x-y)\biggr)^{q'}
\frac{dt}{t}\biggr)^{r'/q'}dy\biggr)^{1/r'}.
\end{align*}

In the same way as in \cite[p.~705]{ChenFanYing}, using
\eqref{eq:sigma-maxLp}, we can check
\begin{align*}
\biggl(\int_\Rn \biggl(\sum_{k\in\Z}
&\int_{2^k}^{2^{k+1}}\biggl(\int_\Rn|g_k(x,t)|d|\sigma_t|(x-y)
\biggr)^{q'}\frac{dt}{t}\biggr)^{r'/q'}dy\biggr)^{1/r'}
\\
&\le C\|\Omega\|_{L^1(\Sn)}
\biggl(\int_\Rn\biggl(\sum_{k\in\Z}
\int_{2^k}^{2^{k+1}}|g_k(x,t)|^{q'}\frac{dt}{t}\biggr)^{r'/q'}dx\biggr)^{1/r'},
\end{align*}
if $(\frac{r'}{q'})'>\gamma'$.
Hence we have, for $1<q'<r'<\gamma q'$,
\begin{equation}\label{eq:mu-j-Lp-4}
\|\tilde\mu_{\Omega,\rho,\phi,\alpha,q,j}^{(b)}f\|_{L^r(\Rn)}
\le C2^{-\alpha(j)j}\|f\|_{\dot F_{rq}^{\alpha}}.
\end{equation}
So, we are done.
\end{enumerate}
\par
\subsection{Proof of Lemma \ref{lem:121108-2}}

By virtue of Plancherel's theorem
and the Fubini theorem, we have
\begin{equation}\label{eq:mu-j-7-7}
\begin{aligned}
\|\tilde\mu_{\Omega,\rho,\phi,\alpha,2,j}^{(b)}f\|_{L^2}^2
&=\sum_{k=-\infty}^{\infty}\int_{\Rn}\left(
\int_{2^k}^{2^{k+1}}|\Psi_{j-k}*\sigma_t*f(x)|^2
\frac{dt}{t\phi(t)^{2\alpha}}\right)\,dx
\\
&\le C
\sum_{k=-\infty}^{\infty}\int_{\Rn}\left(
\int_{2^k}^{2^{k+1}}|\widehat{\sigma_t}(\xi)|^2
\frac{dt}{t\phi(t)^{2\alpha}}\right)|\hat f(\xi)|^2\psi_{j-k}(\xi)^2\,d\xi.
\end{aligned}
\end{equation}

By \eqref{eq:mu-j-7-6}, \eqref{eq:mu-j-7-7} and the
support property of $\psi_{j-k}$, we have
\begin{align*}
\|\tilde\mu_{\Omega,\rho,\phi,\alpha,2,j}^{(b)}f\|_{L^2}^2
&\le C
\|b\|_{\Delta_1}
\sum_{k=-\infty}^{\infty}\int_{\Rn}
\frac{|\xi|^{2}}{\phi(2^k)^{2\alpha-2}}
|\hat f(\xi)|^2\psi_{j-k}(\xi)^2\,d\xi
\\
&\le C
\|b\|_{\Delta_1}
\sum_{\ell=-\infty}^{\infty}\int_{\Rn}
\frac{1}{\phi(2^{-\ell})^2\phi(2^{j-\ell})^{2\alpha-2}}
|\hat f(\xi)|^2\psi_{\ell}(\xi)^2\,d\xi.
\end{align*}
For $j\le 0$, from \eqref{eq:121105-1}, 
it follows 
$
\phi(2^{j-\ell})\le C2^{j/c_1}\phi(2^{-\ell}),
$ 
and for 
$j\ge 0$, from \eqref{eq:121105-1000} we get 
$
\phi(2^{j-\ell})\le 2^{j\log_2c_0}\phi(2^{-\ell}).$
Likewise, we have 
$
\phi(2^{-\ell})
\le 
2^{-\alpha(j)j}\phi(2^{j-\ell}).
$

We need to control the integrand;
first of all,
$$
\frac{1}{\phi(2^{j-l})^{2\alpha-2}}
=
\frac{\phi(2^{j-l})^{2}}{\phi(2^{j-l})^{2\alpha}}
$$
When
$j\ge0$,
we use
\begin{equation*}
\phi(2^{j-\ell})^2\le C2^{2j\log_2 c_0}\phi(2^{-\ell})^2
\end{equation*}
and
\begin{equation*}
\Bigl(\frac{1}{\phi(2^{j-\ell})}\Bigr)^{2\alpha}
\le C\Bigl(\frac{2^{-\alpha(j)j}}{\phi(2^{-\ell})}\Bigr)^{2\alpha}
\le C\frac{2^{-2(\alpha/c_1)j}}{\phi(2^{-\ell})^{2\alpha}}.
\end{equation*}
When $j\le0$,
we use
\begin{equation*}
\phi(2^{j-\ell})^2\le C2^{2j/c_1}\phi(2^{-\ell})^2
\end{equation*}
and
\begin{equation*}
\Bigl(\frac{1}{\phi(2^{j-\ell})}\Bigr)^{2\alpha}
\le C\Bigl(\frac{2^{-\alpha(j)j}}{\phi(2^{-\ell})}\Bigr)^{2\alpha}
\le C\frac{2^{-2\alpha(\log_2c_0)j}}{\phi(2^{-\ell})^{2\alpha}}.
\end{equation*}
So, if $j\ge0$, we have 
\begin{align*}
\|\tilde\mu_{\Omega,\rho,\phi,\alpha,2,j}^{(b)}f\|_{L^2}^2
&\le C
2^{2(\log_2c_0-\alpha/c_1)j}
\sum_{\ell=-\infty}^{\infty}\int_{\Rn}\frac{1}{\phi(2^{-\ell})^{2\alpha}}
|\hat f(\xi)|^2\psi_{\ell}(\xi)^2\,d\xi
\\
&\le C
2^{2(\log_2c_0-\alpha/c_1)j}
\|f\|_{\dot F_{2,2}^\alpha}^2.
\end{align*}
Hence, after incorporating a similar estimate for $j \le 0$, we get
(\ref{eq:mu-j-7-9}).


\subsection{Interpolation and the conclusion of the proof of (i)}
Let
$$
\alpha\in\left(0,
\frac{1}{c_1\log_2c_0}
\cdot
\frac{1/{\tilde p}-1/(2\gamma)}{1/2-1/(2\gamma)}
\cdot
\frac{1/{\tilde q}-1/(2\gamma)}{1/2-1/(2\gamma)}
\right).
$$
By interpolating \eqref{eq:mu-j-7-9} and \eqref{eq:mu-j-Lp-5},
we claim that 
there exists $\delta>0$ such that
\begin{equation}\label{eq:mu-j-Lp-6}
\|\tilde\mu_{\Omega,\rho,\phi,\alpha,q,j}^{(b)}f\|_{L^p(\Rn)}
\le C2^{-\delta|j|}\|f\|_{\dot F_{pq}^{\alpha}}.
\end{equation}

When $p=q=2$, then (\ref{eq:mu-j-Lp-6})
is correct by virtue of (\ref{eq:mu-j-Lp-5}) $(j\ge0)$ 
and (\ref{eq:mu-j-7-9}) $(j<0)$. 
We check next the case  $p\ne2$ and $q\ne2$.
For $j\ge0$, by \eqref{eq:mu-j-Lp-5} we may take $\delta=\alpha(j)=\alpha/c_1$.
For $j\le-1$, we take $1<r_1,r_2<\infty$ and $0<\theta_1,\theta_2<1$
satisfying
\begin{align}
\frac{1}{p}&=\frac{\theta_1}{2}+\frac{1-\theta_1}{r_1},
\label{eq:mu-j-Lp-7-1}
\\
\frac{1}{q}&=\frac{\theta_2}{2}+\frac{1-\theta_2}{r_2}.
\label{eq:mu-j-Lp-7-2}
\end{align}
Note that we have
\[
(p-2)(r_1-2)>0,
(q-2)(r_2-2)>0.
\]
We choose $1<r_1,r_2<\infty$ so that
$$
\tilde p<\tilde{r_1}<2\gamma,
\tilde q<\tilde{r_2}<2\gamma
$$
and then
determine $\theta_1,\,\theta_2$ by the equations \eqref{eq:mu-j-Lp-7-1},
\eqref{eq:mu-j-Lp-7-2}. 
As in the graph,
we can arrange that
\begin{equation}\label{eq:mu-j-Lp-7-3}
\alpha<\frac{\theta_1\theta_2}{c_1\log_2c_0}<\frac1{c_1\log_2c_0}\,
\cdot
\frac{1/{\tilde p}-1/(2\gamma)}{1/2-1/(2\gamma)}\,
\cdot
\frac{1/{\tilde q}-1/(2\gamma)}{1/2-1/(2\gamma)}.
\end{equation}
We shall see that this choice is possible. 
Recall that
$\tilde{p},\tilde{q}<2\gamma$.
Then 
an arithmetic shows
\begin{align*}
\theta_1&=\frac{1/p-1/r_1}{1/2-1/r_1}=\frac{1/p'-1/r_1'}{1/2-1/r_1'}
=\frac{1/\tilde p-1/\tilde{r_1}}{1/2-1/\tilde{r_1}},
\end{align*}
and that
\begin{align*}
\theta_2&=\frac{1/q-1/r_2}{1/2-1/r_2}=\frac{1/q'-1/r_2'}{1/2-1/r_2'}
=\frac{1/\tilde q-1/\tilde{r_2}}{1/2-1/\tilde{r_2}}.
\end{align*}
Assuming that $\tilde{p},\tilde{q}>2$,
we conclude that the parameters $\theta_1$ and $\theta_2$ 
are increasing on $(2,\infty)$ with respect to $\tilde{r_1}$ and
$\tilde{r_2}$ as functions in $\tilde{r_1}$ and $\tilde{r_2}$, respectively.
Hence
\begin{equation*}
\theta_1\theta_2=\frac{1/\tilde p-1/\tilde{r_1}}{1/2-1/\tilde{r_1}}
\cdot
\frac{1/\tilde q-1/\tilde{r_2}}{1/2-1/\tilde{r_2}}<
\frac{1/{\tilde p}-1/(2\gamma)}{1/2-1/(2\gamma)}
\cdot
\frac{1/{\tilde q}-1/(2\gamma)}{1/2-1/(2\gamma)}.
\end{equation*}
Therefore, since
$$
0<\alpha<
\frac{1}{c_1\log_2c_0}
\cdot
\frac{1/{\tilde p}-1/(2\gamma)}{1/2-1/(2\gamma)}
\cdot
\frac{1/{\tilde q}-1/(2\gamma)}{1/2-1/(2\gamma)},
$$ and 
${\tilde p}, {\tilde q}<2\gamma$, we get \eqref{eq:mu-j-Lp-7-3} by choosing 
$r_1$ sufficiently near $2\gamma$ if $p>2$ and $r_1'$ sufficiently near 
$2\gamma$ if $1<p<2$, and by choosing $r_2$ similarly according to $q>2$ or 
$1<q<2$.

Now, interpolating \eqref{eq:mu-j-7-9} and \eqref{eq:mu-j-Lp-5} with
$r=r_1, q=2$, we get
\begin{equation}\label{eq:mu-j-Lp-5-1}
\|\tilde\mu_{\Omega,\rho,\phi,\alpha,2,j}^{(b)}f\|_{L^{p}(\Rn)}\le 
C2^{\bigl(\theta_1(1/c_1-\alpha\log_2c_0)-(1-\theta_1)\alpha\log_2c_0\bigr)j}
\|f\|_{\dot F_{p2}^{\alpha}}.
\end{equation}
We then interpolate \eqref{eq:mu-j-Lp-5} and \eqref{eq:mu-j-Lp-5-1}
with $r=p, q=r_2$. As a consequence, we have
\begin{equation*}
\|\tilde\mu_{\Omega,\rho,\phi,\alpha,q,j}^{(b)}f\|_{L^{p}(\Rn)}
\le C2^{\{\theta_2\bigl(
\theta_1(1/c_1-\alpha\log_2c_0)-(1-\theta_1)\alpha\log_2c_0\bigr)
-(1-\theta_2)\alpha\log_2c_0\}j}\|f\|_{\dot F_{pq}^{\alpha}}.
\end{equation*}
An arithmetic together with (\ref{eq:mu-j-Lp-7-3}) shows that
\begin{equation*}
\theta_2
\bigl(\theta_1(1/c_1-\alpha\log_2c_0)-(1-\theta_1)\alpha\log_2c_0\bigr)
-(1-\theta_2)\alpha\log_2c_0=\theta_1\theta_2/c_1-\alpha\log_2c_0>0.
\end{equation*}
Thus, 
taking $\delta=\min\{\alpha/c_1,\,\theta_1\theta_2/c_1-\alpha\log_2c_0\}$,
we obtain the desired estimate \eqref{eq:mu-j-Lp-6}.

In the case $p=2$ or $q=2$, we can get the desired estimate more simply, 
by applying interpolation once. \par
Thus by \eqref{eq:fracMarSurf-q-1} and
\eqref{eq:mu-j-Lp-6} we obtain
\begin{equation}\label{eq:121106-11}
\|\tilde\mu_{\Omega,\rho,\phi,\alpha,q}^{(b)}f\|_{L^p(\Rn)}
\le \sum_{j\in\Z}\|\tilde\mu_{\Omega,\rho,\phi,\alpha,q,j}^{(b)}f\|_{L^p(\Rn)}
\le C\|f\|_{\dot F_{pq}^{\alpha}}.
\end{equation}
This completes the proof of Theorem \ref{thm:fracMarcSurf-1-b}(i).
\par\smallskip
\subsection{The proof of (ii)}
Below we shall prove Theorem \ref{thm:fracMarcSurf-1-b}(ii). 
By the Schwarz inequality,
we have
\begin{equation*}
\begin{aligned}
|\widehat {\sigma_t}(\xi)|&=\frac{1}{t^{\rho}}\biggl|
\int_{t/2<|y|\le t}e^{-i\phi(|y|)y'\cdot \xi}
\frac{b(|y|)\Omega(y')}{|y|^{n-\rho}}\,dy\biggr|
\\
&=\frac{1}{t^{\rho}}\biggl|\int_{t/2}^{t}\left(
\int_{\Sn}\Omega(y')e^{-i\phi(r)y'\cdot \xi}\,d\sigma(y')
\right)b(r)r^{\rho-1}
dr\biggr|
\\
&\le\biggl(\int_{t/2}^{t}|b(r)|^2\frac{dr}{r}\biggr)^{1/2}
\biggl(\int_{t/2}^{t}\Bigl|\int_{\Sn}
\Omega(y')e^{-i\phi(r)y'\cdot \xi}\,d\sigma(y')\Bigr|^2\frac{dr}r\biggr)^{1/2}
\\
&\le C\|b\|_{\Delta_2}\biggl(\int_{t/2}^{t}\Bigl|\int_{\Sn}
\Omega(y')e^{-i\phi(r)y'\cdot \xi}\,d\sigma(y')\Bigr|^2\frac{dr}r\biggr)^{1/2}.
\end{aligned}
\end{equation*}
Recall 
\begin{equation}\label{eq:mu-j-7-3-1-bb}
W_\Omega=\sup_{\xi \in \Rn \setminus \{0\}}\biggl(
\int_{\Sn\times\Sn}\frac{|\Omega(y'){\Omega(z')}|}
{|(y'-z')\cdot \xi'|^{\beta}}\,d\sigma(y')d\sigma(z')\biggr)^{1/2}.
\end{equation}
Then, by \eqref{eq:mu-j-7-3-1-b} and the doubling condition of $\phi$, we have
\begin{equation}\label{eq:mu-j-7-6-1}
\biggl(\int_{2^k}^{2^{k+1}}|\widehat{\sigma_t}(\xi)|^2
\frac{dt}{t\phi(t)^{2\alpha}}\biggr)^{1/2}
\le\frac{CW_\Omega\|b\|_{\Delta_2}}{|\xi|^{\beta/2}\phi(2^{k})^{\beta/2}\phi(2^{k})^\alpha}.
\end{equation}
By \eqref{eq:mu-j-7-7}, \eqref{eq:mu-j-7-6-1} and the support property of 
$\psi_{j-k}$, we have
\begin{align*}
\|\tilde\mu_{\Omega,\rho,\phi,\alpha,2,j}^{(b)}f\|_{L^2(\Rn)}^2
&\le CW_\Omega^2\|b\|_{\Delta_2}^2\sum_{k=-\infty}^{\infty}
\int_{\Rn}\frac{|\hat f(\xi)|^2\psi_{j-k}(|\xi|)^2}
{|\xi|^{\beta}\phi(2^{k-1})^{\beta}\phi(2^{k+1})^{2\alpha}}\,d\xi
\\
&\le CW_\Omega^2\|b\|_{\Delta_2}^2
\sum_{\ell=-\infty}^{\infty}
\int_{\Rn}\Bigl(\frac{\phi(2^{-\ell})}
{\phi(2^{j-\ell})}\Bigr)^{\beta}
\phi(2^{j-\ell})^{-2\alpha}
|\hat f(\xi)|^2\psi_{\ell}(|\xi|)^2\,d\xi.
\end{align*}
As in the case (i), we have 
$$
\phi(2^{-\ell})\le 2^{-j/c_1}\phi(2^{j-\ell})
$$ 
for $j\ge0$, and 
$$
\phi(2^{-\ell})\le 2^{-j\log_2c_0}\phi(2^{j-\ell})
$$ 
for $j\le0$.   
Similarly, we have 
$$
\phi(2^{j-\ell})\le 2^{j/c_1}\phi(2^{-\ell})
$$ 
for $j\le0$ 
and 
$$
\phi(2^{j-\ell})\le c_0^{j-1}=2^{j\log_2c_0}\phi(2^{-\ell})
$$ 
for $j\ge0$. So, as in the $L^2$-estimate in (i), we obtain
\begin{equation}\label{eq:mu-j-7-9-}
\|\tilde\mu_{\Omega,\rho,\phi,\alpha,2,j}^{(b)}f\|_{L^2(\Rn)}
\le 
\begin{cases}
C2^{-(\beta/(2c_1)+\alpha\log_2c_0)j}W_\Omega\|b\|_{\Delta_2}
\|f\|_{\dot F_{2,2}^\alpha}
&\text{ if } j\ge1,
\\
C2^{-((\beta\log_2c_0)/2+\alpha/c_1)j}W_\Omega\|b\|_{\Delta_2}
\|f\|_{\dot F_{2,2}^\alpha}&\text{ if } j\le0.
\end{cases}
\end{equation}
As for the $L^p$-estimate, since $\alpha<0$, we use 
$\phi(2^{j-\ell})\le c_0^j\phi(2^{-\ell})$ for $j\ge0$ and 
$\phi(2^{j-\ell})\le 2^{j/c_1}\phi(2^{-\ell})$ for $j\le0$.
Hence 
we get, as in the $L^p$-estimate in (i), for any $1<q,r<\infty$ with 
$\tilde r<\gamma\tilde q$ and $j\in\Z$
\begin{equation}\label{eq:mu-j-Lp-5-}
\|\tilde\mu_{\Omega,\rho,\phi,\alpha,q,j}^{(b)}f\|_{L^r(\Rn)}
\le \begin{cases}
C2^{-(\alpha/c_1)j}W_\Omega\|b\|_{\Delta_2}\|f\|_{\dot F_{rq}^{\alpha}}
\text{ for }j\le0,
\\
C2^{-(\alpha \log_2c_0)j}W_\Omega\|b\|_{\Delta_2}\|f\|_{\dot F_{rq}^{\alpha}}
\text{ for }j\ge0.
\end{cases}
\end{equation}

It follows that, for
$$\alpha\in
\left(-\frac{2\beta}{c_1\log_2c_0}
\cdot
\frac{1/{\tilde p}-1/(2\gamma)}{1-1/\gamma}
\cdot
\frac{1/{\tilde q}-1/(2\gamma)}{1-1/\gamma},0\right), 
$$
there still exists $\delta>0$ such that
\begin{equation}\label{eq:mu-j-Lp-6-1}
\|\tilde\mu_{\Omega,\rho,\phi,\alpha,q,j}^{(b)}f\|_{L^p(\Rn)}
\le C2^{-\delta|j|}\|f\|_{\dot F_{pq}^{\alpha}},
\end{equation}
by using \eqref{eq:mu-j-Lp-5-} in the case $j\le0$, and
interpolating \eqref{eq:mu-j-7-9-} and
\eqref{eq:mu-j-Lp-5-}
in the case $j>0$, as in the case (i).

Thus by \eqref{eq:fracMarSurf-q-1} and
\eqref{eq:mu-j-Lp-6-1} we obtain
\begin{equation*}
\|\tilde\mu_{\Omega,\rho,\phi,\alpha,q}^{(b)}f\|_{L^p(\Rn)}
\le 
\sum_{j\in\Z}\|\tilde\mu_{\Omega,\rho,\phi,\alpha,q,j}^{(b)}f\|_{L^p(\Rn)}
\le C
W_\Omega\|b\|_{\Delta_2}\|f\|_{\dot F_{pq}^{\alpha}}.
\end{equation*}
This completes the proof of Theorem \ref{thm:fracMarcSurf-1-b}(ii). 
\par\smallskip
\subsection{Proof of (iii)}
We proceed to show (iii). Let $\Omega\in L\log L(\Sn)$. 
We normalize $\Omega$ to have
$\|\Omega\|_{L\log L(\Sn)}=1$.
Then, as in \cite[p.698-p.699]{AACPan}, 
there is a subset $\Lambda\subset \N \cup \{0\}$
and 
a sequence of functions $\{\Omega_m; m\in\Lambda\}$
 satisfying
$0 \in \Lambda$ and the following conditions;
\begin{equation}\label{eq:fracMar-8-2}
\int_{\Sn}\Omega_m(y')\,d\sigma(y')=0;
\end{equation}
\begin{equation}\label{eq:fracMar-8-4}
\Omega(x')=\sum_{m\in\Lambda}\Omega_m(x');
\end{equation}
\begin{equation}\label{eq:fracMar-8-5}
\|\Omega_0\|_{L^2(\Sn)}
+
\sum_{m\in\Lambda}m
\|\Omega_m\|_{L^1(\Sn)}\le C
\|\Omega\|_{L\log L(\Sn)}.
\end{equation}
Indeed, we just let
\[
\Lambda=\{m \in \N\,:\,\sigma\{2^{m-1}<|\Omega| \le 2^m\}>2^{-4m}\}
\]
and define
\[
\Omega_m(x)=\Omega(x)\chi_{\{2^{m-1}<|\Omega| \le 2^m\}}(x)-
\frac{1}{\sigma(\Sn)}
\int_{2^{m-1}<|\Omega(y)| \le 2^m}\Omega(y)\,d\sigma(y),
\]
\[
\Omega_0(x)=\Omega(x)-\sum_{m \in \Lambda \cap \N}\Omega_m(x).
\]

Now for $m\in\Lambda$,
by observing the proof of the case (i), 
we choose $\theta_1$ and $\theta_2$ very close 
to $$
4\frac{1/\tilde{p}-1/(2\gamma)}{1-1/\gamma}
\cdot
\frac{1/\tilde{q}-1/(2\gamma)}{1-1/\gamma}
$$ 
so that $\delta=\alpha/c_1$ for small $\alpha>0$. 
For large $m$, setting $\alpha=1/m$, we obtain
\begin{equation}\label{eq:fracMar-8-6}
\|\mu_{\Omega_m,\rho,\phi,\alpha,q,j}^{(b)}f\|_{L^p(\Rn)}
\le C2^{-|j|/m}\|\Omega_m\|_{L^1(\Sn)}\|f\|_{\dot F_{pq}^{1/m}},\ j\in\Z.
\end{equation}
Next, from  $\Omega_m\in L^2(\Sn)$ it follows that $\Omega_m$ satisfies the
condition in Theorem \ref{thm:fracMarcSurf-1-b}(ii) for any $\beta<1/2$. Fix
$0<\beta<{1/2}$, and $\alpha_0>0$ with 
$$
\alpha_0<
\left(0,\min\left\{
\frac{2\beta}{c_1\log_2c_0}
\cdot
\frac{1/\tilde{p}-1/(2\gamma)}{1-1/\gamma}
\cdot
\frac{1/\tilde{q}-1/(2\gamma)}{1-1/\gamma},
\frac{\rho}{\log_2c_0}
\right\}
\right).
$$ 
Let also
$$
\delta_0
=\min\left\{\frac{\alpha_0}{c_1},
\frac{\beta\theta_1\theta_2}{2c_1}+\alpha_0\log_2c_0
\right\}
$$
in the proof of the case (ii). Then we obtain
\begin{equation}\label{eq:fracMar-8-7}
\|\mu_{\Omega_m,\rho,\phi,-\alpha_0,q,j}^{(b)}f\|_{L^p(\Rn)}
\le C2^{-\delta_0|j|}\|\Omega_m\|_{L^2(\Sn)}\|f\|_{\dot F_{pq}^{-\alpha_0}},
\ j\in\Z.
\end{equation}
Since 
$\frac{\alpha_0}{1/m+\alpha_0}+\bigl(1-\frac{\alpha_0}{1/m+\alpha_0}\bigr)=1$
and 
$\frac1m\cdot \frac{\alpha_0}{1/m+\alpha_0}-\alpha_0
\bigl(1-\frac{\alpha_0}{1/m+\alpha_0}
\bigr)=0,
$
an interpolation between \eqref{eq:fracMar-8-6} and \eqref{eq:fracMar-8-7}
yields
\begin{align*}
&\|\mu_{\Omega_m,\rho,\phi,0,q,j}^{(b)}f\|_{L^p(\Rn)}
\\
&\le C2^{-(\alpha_0/(1+m\alpha_0)+\delta_0/(1+m\alpha_0))|j|}
\|\Omega_m\|_{L^1(\Sn)}^{\alpha_0/(1/m+\alpha_0)}
\|\Omega_m\|_{L^2(\Sn)}^{1/(1+m\alpha_0)}\|f\|_{\dot F_{pq}^{0}}
\\
&\le C2^{-|j|/m}2^{4/\alpha_0}\|f\|_{\dot F_{pq}^{0}},
\ j\in\Z.
\end{align*}
Thus, summing the above estimate up, we obtain
\begin{equation}\label{eq:fracMar-8-9}
\|\mu_{\Omega_m,\rho,\phi,0,q}^{(b)}f\|_{L^p(\Rn)}
\le \frac{C}{1-2^{-1/m}}\|f\|_{\dot F_{pq}^{0}}\le C\,m\|f\|_{\dot F_{pq}^{0}}.
\end{equation}
Combining \eqref{eq:fracMar-8-9} with \eqref{eq:fracMar-8-4} and
\eqref{eq:fracMar-8-5} and the definition of $\mu_{\Omega,0,\rho,q}^{(b)}$,
we obtain the desired estimate
\begin{equation*}
\|\mu_{\Omega,\rho,\phi,0,q}^{(b)}f\|_{L^p(\Rn)}
\le C\|\Omega\|_{L\log L(\Sn)}\|f\|_{\dot F_{pq}^{0}}.
\end{equation*}
Thus, we are done.

\qed
\section{Proof of Theorem \ref{thm:fracMarcSurf-1} }
\label{section:4}

Here we shall relax the condition on $\alpha$
by taking advantage of a new condition
on $\phi$.
We use the notations in the proof of Theorem \ref{thm:fracMarcSurf-1-b}, by 
setting $b(t)\equiv1$ and $\gamma=\infty$.  
Using \eqref{eq:doubling} and \eqref{eq:uniform}, we apply Theorem 
\ref{thm:fracMarcSurf-1-b}(i), and obtain the conclusion of 
Theorem \ref{thm:fracMarcSurf-1}(i). 

We go to the proof of (ii). First
\begin{align}
\widehat {\sigma_t}(\xi)&=\frac{1}{t^\rho}\int_{B(t) \setminus B(t/2)}
\frac{\Omega(y')}{|y|^{n-\rho}}e^{-i\phi(|y|)y'\cdot \xi}dy
\label{eq:FTsigma-b=1}
\\
&=\int_{\Sn}\Omega(y')\frac{1}{t^\rho}
\left(\int_{t/2}^{t}
e^{-i\phi(|y|)y'\cdot \xi}r^{\rho-1}\,dr\right)\,d\sigma(y')\notag.
\end{align}
With change of variables we get
\begin{align}
B(t,\xi)
&:=\frac{1}{t^\rho}\int_{t/2}^{t}e^{-i\phi(|y|)y'\cdot \xi}r^{\rho-1}\,dr
\nonumber\\
&=\frac{1}{t^\rho}\int_{\phi(t/2)}^{\phi(t)}e^{-isy'\cdot \xi}
\frac{\phi^{-1}(s)^{\rho}}{\phi^{-1}(s)\phi'(\phi^{-1}(s))}\,ds
\label{eq:FTsigma-b=1-2}
\\
&=\frac{1}{t^\rho}\int_{\phi(t/2)}^{\phi(t)}e^{-isy'\cdot \xi}
\frac{\phi^{-1}(s)^{\rho}}{s}
\cdot
\frac{\phi(\phi^{-1}(s))}{\phi^{-1}(s)\phi'(\phi^{-1}(s))}\,ds.
\nonumber\\
&=\frac{1}{t^\rho}\int_{\phi(t/2)}^{\phi(t)}e^{-isy'\cdot \xi}
\frac{\phi^{-1}(s)^{\rho}}{s}
\cdot
\varphi(\phi^{-1}(s))\,ds.
\nonumber
\end{align}
Suppose now that $t\phi'(t)$ is increasing on $\R_+$. Then 
$\phi^{-1}(s)\phi'(\phi^{-1}(s))$ is also increasing. So by applying 
the second mean value theorem to the real part of the expression 
\eqref{eq:FTsigma-b=1-2}, we see that there exists $u$ with 
$\phi(t/2)<u<\phi(t)$ such that 
\begin{equation*}
\operatorname{Re}B(t,\xi)
=\frac{1}{t^\rho\phi^{-1}(\phi(t/2))\phi'(\phi^{-1}(\phi(t/2))}
\int_{\phi(t/2)}^{u}
{\rm Re}(e^{-isy'\cdot \xi}){\phi^{-1}(s)^{\rho}}\,ds.
\end{equation*}
Since $\phi^{-1}(s)^{\rho}$ is increasing, we have 
\begin{align*}
|\operatorname{Re}B(t,\xi)|
&\le \frac{\phi^{-1}(u)^{\rho}}
{t^\rho\phi^{-1}(\phi(t/2))\phi'(\phi^{-1}(\phi(t/2))|y'\cdot\xi|}
\\
&\le\frac{\phi^{-1}(\phi(t))^{\rho}}{t^\rho}
\cdot
\frac{\phi(\phi^{-1}(\phi(t/2)))}
{\phi^{-1}(\phi(t/2))\phi'(\phi^{-1}(\phi(t/2))}
\cdot
\frac1{\phi(t/2)|y'\cdot\xi|}
\\
&=\frac{\phi^{-1}(\phi(t))^{\rho}}{t^\rho}
\cdot
\frac{\phi(t/2)}
{t/2 \cdot \phi'(t/2)}
\cdot
\frac1{\phi(t/2)|y'\cdot\xi|}\\
&=\frac{\phi^{-1}(\phi(t))^{\rho}\varphi(t/2)}{t^\rho}
\cdot
\frac1{\phi(t/2)|y'\cdot\xi|}
\le C\frac{c_0\|\varphi\|_\infty}{\phi(t)|y'\cdot\xi|}. 
\end{align*}
After estimating $\operatorname{Im}B(t,\xi)$ 
in a similar manner, we obtain 
\begin{equation}\label{eq:FTsigma-b=1-4}
|B(t,\xi)|\le C\frac{c_0\|\varphi\|_\infty}{\phi(t)|y'\cdot\xi|}.
\end{equation}
In the case $t\phi'(t)$ is decreasing or $\varphi(t)$ is monotonic, we get the 
same estimate \eqref{eq:FTsigma-b=1-4} in a similar way. Clearly, we have 
$|B(t,\xi)|\le 1/\rho$, and hence,
for any $0<\beta\le1$,
$|B(t,\xi)|\le \dfrac{C}{(\phi(t)|y'\cdot\xi|)^\beta}.$
By \eqref{eq:FTsigma-b=1} we get
\begin{equation}
|\widehat {\sigma_t}(\xi)|
\le C\left(\int_{\Sn}\frac{|\Omega(y')}{|y'\cdot\xi'|^\beta}d\sigma(y')
\right)
\frac1{(\phi(t)|\xi|)^\beta}. 
\end{equation}
Now the rest of the proof is the same as that of the case (i). 
\par\smallskip

This completes the proof of Theorem \ref{thm:fracMarcSurf-1}.
\qed


\section{Proof of Proposition \ref{lem:Triebel-Liz1}}
\label{section:5}
 
The part is an appendix of the present paper,
where we prove Proposition \ref{lem:Triebel-Liz1}.
Let $\psi \in {\mathcal S}(\Rn)$ be chosen so that
\begin{equation*}
\chi_{B(1)} \le \psi \le \chi_{B({a^{1/3}})}.
\end{equation*}
Define
\begin{equation}\label{eq:str-1}
\varphi_k(\xi)=\psi(a_{k}^{-1}\xi)-\psi(a_{k-1}^{-1}\xi)
\quad(\xi \in \Rn).
\end{equation}

Notice that $\supp \varphi_k\subset\{\xi\in\Rn;\,a_{k-1}\le |\xi|\le 
a^{1/3}a_{k}\}$ $(k\in\Z)$ and that $\varphi_k(\xi)=1$ 
on $\{a^{1/3}a_{k-1} \le |\xi| \le a_{k}\}$. Let 
\begin{equation}\label{eq:str-2}
\Phi_k=\F^{-1}\varphi_k.
\end{equation} 
Then, we see that $\{\Phi_k\}_{k\in\Z}$ is a partition of unity adapted to 
$\{a_k\}_{k\in\Z}$. 
Similarly, taking $\psi$ so that 
\begin{equation*}
\chi_{B(a^{-1/3})} \le \psi \le \chi_{B({1})}, 
\end{equation*}
and setting 
\begin{equation*}
\varphi_k(\xi)=\psi(a_{k+1}^{-1}\xi)-\psi(a_{k}^{-1}\xi)
\quad (\xi \in \Rn),
\end{equation*}
we obtain 
another partition of unity $\{\Psi_k\}_{k\in\Z}$ adapted to 
$\{a_k\}_{k\in\Z}$ 
satisfying $\supp \widehat{\Psi_k}\subset
\{\xi\in\Rn;\,a_{k}/a^{1/3}\le |\xi|\le a_{k+1}\}$ $(k\in\Z)$ and 
$\widehat{\Psi_k}(\xi)=1$ on $\{a_{k} \le |\xi| \le a_{k+1}/a^{1/3}\}$. 
Note that $\{a_{k} \le |\xi| \le a^{2/3}a_{k}\}
\subset\{a_{k} \le |\xi| \le a_{k+1}/a^{1/3}\}$. 
Let us take a function $\Theta \in {\mathcal S}$ so that
${\supp}({\mathcal F}\Theta) \subset B(a^{1/3}/2-1/2)$.
Consider
\[
f_k(x)=f_k(x_1,x_2,\cdots,x_n)=\exp
\left(i\frac{({a^{1/3}}+a^{2/3})}{2}a_k x_1\right)\Theta(a_kx)
\quad (x \in \Rn).
\]
Then we have
\[
{\mathcal F}f_k(\xi)
=a_k{}^{-n}{\mathcal F}\Theta\left(\frac{\xi}{a_k}
-\frac{({a^{1/3}}+a^{2/3})}{2}\mathbf e_1\right)
\quad(\xi \in \Rn),
\]
where $\mathbf e_1=(1,0,\dots,0))$. 
It follows that 
$\supp{\mathcal F}f_k\subset \{a^{1/3}a_k\le|\xi|\le a^{2/3}a_k\}$. 
Hence we have 
\begin{equation*}
\Phi_{j}*f_k(x)=\delta_{(j-1)k}f_k(x)\ \text{ and }
\ \Psi_{j}*f_k(x)=\delta_{jk}f_k(x),
\end{equation*}
where
\[
\delta_{jk}
=
\begin{cases}
1&(j=k),\\
0&(j \ne k)
\end{cases}
\]
for $j,k \in \Z$.
\[
\|f_k\|_{F^{\alpha,\{\Phi_k\}_{k\in\Z}}_{pq}}
=
{a_{k+1}}^{\alpha}
\bigl\|f_k\bigr\|_{L^p(\Rn)}
=
{a_{k+1}}^{\alpha}
\bigl\|\Theta(a_k\cdot)\bigr\|_{L^p(\Rn)}
\]
and
\[
\|f_k\|_{F^{\alpha,\{\Psi_k\}_{k\in\Z}}_{pq+}}
=
{a_{k}}^{\alpha}
\bigl\|f_k\bigr\|_{L^p(\Rn)}
=
{a_{k}}^{\alpha}
\bigl\|\Theta(a_k\cdot)\bigr\|_{L^p(\Rn)}.
\]
Since the two norms are assumed equivalent,
we obtain
\[
\frac{a_{k+1}}{a_k} \le C_0
\]
for some $C_0>1$. Since $\frac{a_{k+1}}{a_k}\ge a$, we have $C_0>a$. 

Thus we have proved the first part of our proposition. 
We proceed to the second part.
 Let $\{a_k\}_{k\in\Z}$ be a lacunary sequence of 
positive numbers with $1<a\le a_{k+1}/a_{k}\le C_0$ $(k\in\Z)$, and 
$\{\Phi_k\}_{k\in\Z}$ be a partition of unity adapted to $\{a_k\}_{k\in\Z}$. 

Now we can define the classical homogeneous Triebel-Lizorkin spaces as follows:
 Let $\psi \in {\mathcal S}(\Rn)$ be chosen so that
$\chi_{B(a^{2/3})} \le \psi \le \chi_{B({a})}.$
Define
\[
\varphi_k(\xi)=\psi(a^{-k}\xi)-\psi(a^{-k+1}\xi)
\quad (\xi \in \Rn).
\]
Notice that $\varphi_k(\xi)=1$
on $\{a^{k} \le |\xi| \le a^{k+\frac23}\}$.

Define
\[
\|f\|_{\dot F^\alpha_{pq}}
=
\Bigl\|
\Bigl(\sum_{j=-\infty}^\infty
|a^{\alpha j}{\mathcal F}^{-1}\varphi_j*f|^q\Bigr)^{1/q}
\Bigr\|_{L^p}
\]

Let us prove 
\[
\|f\|_{\dot F^{\alpha,\{\Phi_k\}_{k\in\Z}}_{pq}} 
\le C \|f\|_{\dot F^\alpha_{pq}}.
\]
For each $k\in\Z$,
we choose $m_k\in\Z$ so that
\[
a^{m_k} \le a_k < a^{m_k+1}.
\]
Combining with $a^{m_k+1}\le aa_k\le a_{k+1}$, we get 
$m_{k+1}\ge m_k+1$. 
And combining with $a_{k+1}/a_k\le C_0$, 
we have $m_{k+1}-m_k\le 1+\log_a C_0$.
Furthermore we have
\[
\Phi_k=\Phi_k*\sum_{l=m_{k-1}}^{m_{k+1}+1}\varphi_l.
\]
Consequently, we obtain
\begin{align*}
\|f\|_{\dot F^{\alpha,\{\Phi_k\}_{k\in\Z}}_{pq}}
&=
\biggl\|\biggl(\sum_{k\in\Z}{a_{k}}^{\alpha q}
|\Phi_{k}*f|^q\biggr)^{1/q}\biggr\|_{L^p(\Rn)}
\\
&=\biggl\|\biggl(\sum_{k\in\Z}{a_{k}}^{\alpha q}
\biggl|\Phi_k*\sum_{l=m_{k-1}}^{m_{k+1}+1}\varphi_l*f\biggr|^q
\biggr)^{1/q}\biggr\|_{L^p(\Rn)}\\
&\le
\biggl\|\biggl(\sum_{k\in\Z}{a_{k}}^{\alpha q}
\biggl|\Phi_k*\sum_{l=m_{k-1}}^{m_{k+1}+1}\varphi_l*f\biggr|^q
\biggr)^{1/q}\biggr\|_{L^p(\Rn)}.
\end{align*}
We now invoke the Plancherel-Polya-Nikolskij inequality:
We have
\[
\left|\sum_{l=m_{k-1}}^{m_{k+1}+1}\varphi_l*f(x)\right|
\le C(1+|a^{m_{k+1}+1}(x-y)|)^n
M\left[\sum_{l=m_{k-1}}^{m_{k+1}+1}\varphi_l*f\right](y).
\]
Using Plancherel's theorem,
the assumption 
$|\xi^\beta\partial^{\beta}\widehat{\Phi_k}(\xi)|\le C_\beta$
for all $\beta$ 
and that
$\supp\widehat\Phi_k\subset \{a_{k-1}\le|\xi|\le a_{k+1}\}$, 
we get 
\begin{align*}
\lefteqn{
\int_{\Rn}(1+|a^{m_{k+1}+1}x|)^n|\Phi_k(x)|\,dx
}\\
&\le C
\int_{\Rn}(1+|a_{k+1}x|)^n|\Phi_k(x)|\,dx
\\
&\le C
\left(\int_{\Rn}(1+|a_{k+1}x|)^{4n}|\Phi_k(x)|^2\,dx\right)^{1/2}
\left(\int_{\Rn}(1+|a_{k+1}x|)^{-2n}\,dx\right)^{1/2}
\\
&\le C{a_k}^{-n/2}
\left(\int_{\Rn}
\left(
|\widehat{\Phi}_k(\xi)|^2
+a_{k+1}{}^{4n}
|{\nabla^{n}\widehat{\Phi}_k(\xi)}|^2
\right)\,d\xi\right)^{1/2}
\\
&\le C{a_k}^{-n/2}
\left(\int_{a_{k-1}\le|\xi|\le a_{k+1}}
\left(
|\widehat{\Phi}_k(\xi)|^2
+a_{k+1}{}^{4n}|\xi|^{-4n}
\right)\,d\xi\right)^{1/2}
\\
&\le C.
\end{align*}
Hence, it follows that
\begin{align*}
\|f\|_{\dot F^{\alpha,\{\Phi_k\}_{k\in\Z}}_{pq}}
&\le C
\biggl\|\biggl(\sum_{k\in\Z}{a_{k}}^{\alpha q}
M\biggl[\biggl|\sum_{l=m_{k-1}}^{m_{k+1}+1}\varphi_l*f\biggr|\biggr]^q
\biggr)^{1/q}\biggr\|_{L^p(\Rn)}.
\end{align*}
By the Fefferman-Stein vector-valued maximal inequality
(see \cite{FeSt}),
we obtain
\begin{align*}
\|f\|_{\dot F^{\alpha,\{\Phi_k\}_{k\in\Z}}_{pq}}
&\le C
\biggl\|\biggl(\sum_{k\in\Z}{a_{k}}^{\alpha q}
\biggl|\sum_{l=m_{k-1}}^{m_{k+1}+1}\varphi_l*f\biggr|^q
\biggr)^{1/q}\biggr\|_{L^p(\Rn)}.
\end{align*}
If we use
$(a_1+a_2+\cdots+a_N)^q \le N^{q}(a_1{}^q+a_2{}^q+\cdots+a_N{}^q)$,
then we obtain
\begin{align*}
\|f\|_{\dot F^{\alpha,\{\Phi_k\}_{k\in\Z}}_{pq}}
&\le C
\biggl\|\biggl(\sum_{k\in\Z}{a_{k}}^{\alpha q}(m_{k+1}-m_{k-1}+2)^{q}
\sum_{l=m_{k-1}}^{m_{k+1}+1}\left|\varphi_l*f\right|^q
\biggr)^{1/q}\biggr\|_{L^p(\Rn)}.
\end{align*}
Noting $m_{k+1}-m_k\ge 1$, 
$m_{k+1}-m_k\le 1+\log_a C_0$
and that
$a^{m_{k-1}+1}\le a^{m_k} \le a_k < a^{m_k+1}\le a^{m_{k-1}+1+[\log_aC_0]}$,
we conclude
\[
\|f\|_{\dot F^{\alpha,\{\Phi_k\}_{k\in\Z}}_{pq}}
\le C
\biggl\|\biggl(\sum_{k\in\Z}{a_{k}}^{\alpha q}
\sum_{l=m_{k-1}}^{m_{k-1}+2+[2\log_aC_0]}\left|\varphi_l*f\right|^q
\biggr)^{1/q}\biggr\|_{L^p(\Rn)}
\le C
\|f\|_{F^\alpha_{pq}}.
\]

Let us prove the reverse inequality.
For each $k\in\Z$,
we can choose $\ell_k\in\Z$ so that
\[
a_{\ell_k} \le a^{k-1} \le a^{k+1} \le a_{\ell_k+3}.
\]
Then we have
\[
\varphi_k=\varphi_k(\Phi_{\ell_k}+\Phi_{\ell_k+1}
+\Phi_{\ell_k+2}+\Phi_{\ell_k+3}).
\]
Notice that
\[
\sup_{l \in \Z}\sharp
\{k\,:\,\ell_k=l\}\le 3\log_aC_0
\]
because $a_{l+3}/a_l \le C_0{}^3$.
Thus, it follows that
\begin{align*}
\|f\|_{\dot F^\alpha_{pq}}
&=
\biggl\|
\biggl(\sum_{k=-\infty}^\infty
|a^{\alpha k}{\mathcal F}^{-1}\varphi_k*f|^q\biggr)^{1/q}
\biggr\|_{L^p}\\
&=
\biggl\|
\biggl(\sum_{k=-\infty}^\infty
|a^{\alpha k}{\mathcal F}^{-1}\varphi_k*({\mathcal F}^{-1}\Phi_{\ell_k}+
\cdots+
{\mathcal F}^{-1}\Phi_{\ell_k+3})*f|^q\biggr)^{1/q}
\biggr\|_{L^p}\\
&\le C
\biggl\|
\biggl(\sum_{k=-\infty}^\infty
M[|a^{\alpha k}({\mathcal F}^{-1}\Phi_{\ell_k}+
\cdots+
{\mathcal F}^{-1}\Phi_{\ell_k+3})*f|]^q\biggr)^{1/q}
\biggr\|_{L^p}.
\end{align*}
Again by the Fefferman-Stein vector-valued maximal inequality
(see \cite{FeSt}),
we obtain
\begin{align*}
\|f\|_{\dot F^\alpha_{pq}}
&\le C
\biggl\|
\biggl(\sum_{k=-\infty}^\infty
|a^{\alpha k}({\mathcal F}^{-1}\Phi_{\ell_k}+
\cdots+
{\mathcal F}^{-1}\Phi_{\ell_k+3})*f|^q\biggr)^{1/q}
\biggr\|_{L^p}\\
&\le C
\biggl\|
\biggl(\sum_{k=-\infty}^\infty
\biggl(
|a_{\ell_k}{}^{\alpha }{\mathcal F}^{-1}\Phi_{\ell_k}*f|
+
\cdots
+
|a_{\ell_k+3}{}^{\alpha}{\mathcal F}^{-1}\Phi_{\ell_k+3}*f|
\biggr)^q\biggr)^{1/q}
\biggr\|_{L^p}\\
&\le C
\|f\|_{\dot F^{\alpha,\{\Phi_k\}_{k\in\Z}}_{pq}}.
\end{align*}
This completes the proof of our proposition.

 
\par\vspace{1cm}


\end{document}